\documentclass{amsart}
\usepackage{amssymb}
\usepackage{mathrsfs}
\usepackage[all]{xy}

\RequirePackage{color}
\definecolor{myred}{rgb}{0.75,0,0}
\definecolor{mygreen}{rgb}{0,0.5,0}
\definecolor{myblue}{rgb}{0,0,0.65}

\RequirePackage{ifpdf}
\ifpdf
  \IfFileExists{pdfsync.sty}{\RequirePackage{pdfsync}}{}
  \RequirePackage[pdftex,
   colorlinks = true,
   urlcolor = myblue, % \href{...}{...} external (URL)
   citecolor = mygreen, % \cite{...}
   linkcolor = myblue, % \ref{...} and \pageref{...}
   pdfusetitle,
   pdfauthor = {{Pramod Achar, Anthony Henderson, Daniel Juteau and Simon Riche}},
   pdfpagemode=None,
   bookmarksopen=true]{hyperref}
\else
  \RequirePackage[hypertex]{hyperref}
\fi

\usepackage{fullpage}

\usepackage{graphicx}
\usepackage{caption}
\usepackage{subcaption}
\usepackage{longtable}
\usepackage{supertabular}

\newcommand{\mgmainthm}{Theorem~1.1}
\newcommand{\mgthmcusp}{Theorem~1.5}
\newcommand{\mglsylow}{\S 4.3}
\newcommand{\mgmincuspdatum}{Proposition~5.1}
\newcommand{\mgindepk}{Remark~1.2}
\newcommand{\mgrathergood}{Lemmas~6.2--6.4}
\newcommand{\mgeasyrathergood}{Lemma~6.3}

\newcommand{\mgeqnweylisom}{Eq.~(4.1)}
\newcommand{\mgdisjoint}{Theorem~2.5}
\newcommand{\mgmackey}{Theorem~2.2}
\newcommand{\mgbigenough}{Proposition~3.2}
\newcommand{\mgsemisimple}{Remark~7.2}

\newcommand{\C}{\mathbb{C}}
\newcommand{\Ql}{\mathbb{Q}_\ell}

\newcommand{\Qlb}{\overline{\mathbb{Q}}_\ell}
\newcommand{\Z}{\mathbb{Z}}

\newcommand{\K}{\mathbb{K}}
\newcommand{\F}{\mathbb{F}}
\renewcommand{\O}{\mathbb{O}}
\newcommand{\E}{\mathbb{E}}
\newcommand{\bk}{\Bbbk}

\newcommand{\fg}{\mathfrak{g}}
\newcommand{\fz}{\mathfrak{z}}

\newcommand{\cN}{\mathscr{N}}
\newcommand{\cO}{\mathscr{O}}

\newcommand{\fL}{\mathfrak{L}}
\newcommand{\fN}{\mathfrak{N}}
\newcommand{\fM}{\mathfrak{M}}

\newcommand{\Db}{D^{\mathrm{b}}}
\newcommand{\Perv}{\mathsf{Perv}}

\newcommand{\Rep}{\mathsf{Rep}}

\newcommand{\For}{\mathrm{For}}
\newcommand{\bT}{\mathbb{T}}

\newcommand{\Res}{\mathbf{R}}
\newcommand{\Ind}{\mathbf{I}}

\newcommand{\cL}{\mathcal{L}}

\newcommand{\IC}{\mathcal{IC}}
\newcommand{\cF}{\mathcal{F}}
\newcommand{\cG}{\mathcal{G}}
\newcommand{\cH}{\mathcal{H}}

\newcommand{\ubk}{\underline{\bk}}

\newcommand{\cE}{\mathcal{E}}

\newcommand{\cD}{\mathcal{D}}

\newcommand{\simto}{\xrightarrow{\sim}}
\DeclareMathOperator{\End}{End}
\DeclareMathOperator{\Hom}{Hom}
\DeclareMathOperator{\Ext}{Ext}

\DeclareMathOperator{\Irr}{Irr}

\newcommand{\id}{\mathrm{id}}
\newcommand{\cusp}{\mathrm{cusp}}

\newcommand{\reg}{\mathrm{reg}}

\makeatletter
\def\lotimes{\@ifnextchar_{\@lotimessub}{\@lotimesnosub}}
\def\@lotimessub_#1{\mathchoice{\mathbin{\mathop{\otimes}^L}_{#1}}%
  {\otimes^L_{#1}}{\otimes^L_{#1}}{\otimes^L_{#1}}}
\def\@lotimesnosub{\mathbin{\mathop{\otimes}^L}}
\makeatother

\newcommand{\GL}{\mathrm{GL}}

\newcommand{\SL}{\mathrm{SL}}
\newcommand{\Sp}{\mathrm{Sp}}

\newcommand{\Spin}{\mathrm{Spin}}
\newcommand{\SO}{\mathrm{SO}}

\newcommand{\fS}{\mathfrak{S}}

\newcommand{\Part}{\mathrm{Part}}
\newcommand{\Bipart}{\mathrm{Bipart}}

\newcommand{\sm}{\mathsf{m}}

\newcommand{\uBipart}{\underline{\Bipart}}

\newcommand{\zs}{0\text{-}}
\newcommand{\zcusp}{{\zs\cusp}}
\renewcommand{\sup}{{\text{\rm sup}}}
\newcommand{\super}{\text{\rm super-}}

\newcommand{\supercusp}{\text{\rm s-cusp}}

\newcommand{\pt}{\mathrm{pt}}
\newcommand{\Serre}{\mathfrak{S}}

\newtheorem*{thm*}{Theorem}
\numberwithin{equation}{section}
\newtheorem{thm}{Theorem}[section]
\newtheorem{lem}[thm]{Lemma}
\newtheorem{prop}[thm]{Proposition}
\newtheorem{cor}[thm]{Corollary}

\newtheorem{conj}[thm]{Conjecture}

\theoremstyle{definition}

\theoremstyle{remark}
\newtheorem{rmk}[thm]{Remark}

\DeclareFontFamily{U}{mathx}{\hyphenchar\font45}
\DeclareFontShape{U}{mathx}{m}{n}{
      <5> <6> <7> <8> <9> <10>
      <10.95> <12> <14.4> <17.28> <20.74> <24.88>
      mathx10
      }{}
\DeclareSymbolFont{mathx}{U}{mathx}{m}{n}
\DeclareFontSubstitution{U}{mathx}{m}{n}
\DeclareMathAccent{\widebar}{0}{mathx}{"73}

\title{Constructible sheaves on nilpotent cones\\ in rather good characteristic}

\author{Pramod N. Achar}
\address{Department of Mathematics\\
  Louisiana State University\\
  Baton Rouge, LA 70803\\
  U.S.A.}
\email{pramod@math.lsu.edu}

\author{Anthony Henderson}
\address{School of Mathematics and Statistics\\
  University of Sydney, NSW 2006\\
  Australia}
\email{anthony.henderson@sydney.edu.au}

\author{Daniel Juteau}
\address{Laboratoire de Math\'ematiques Nicolas Oresme\\
  Universit\'e de Caen, BP 5186\\
  14032 Caen Cedex\\ 
  France}
\email{daniel.juteau@unicaen.fr}

\author{Simon Riche}
\address{Universit{\'e} Blaise Pascal - Clermont-Ferrand II, Laboratoire de Math{\'e}matiques, CNRS, UMR 6620, Campus universitaire des C{\'e}zeaux, F-63177 Aubi{\`e}re Cedex, France
}
\email{simon.riche@math.univ-bpclermont.fr}


\thanks{P.A. was supported by NSA Grant No.~H98230-15-1-0175.  A.H. was supported by ARC Future Fellowship Grant No.~FT110100504. D.J. and S.R. were supported by ANR Grant 
No.~
ANR-13-BS01-0001-01. 
}

\begin{document}

\begin{abstract}
We study some aspects of modular generalized Springer theory for a complex reductive group $G$ with coefficients in a field $\bk$ under the assumption that the characteristic $\ell$ of $\bk$ is rather good for $G$, i.e., $\ell$ is good and does not divide the order of the component group of the centre of $G$.  We prove a comparison theorem relating the characteristic-$\ell$ generalized Springer correspondence to the characteristic-$0$ version.  We also consider Mautner's characteristic-$\ell$ `cleanness conjecture'; we prove it in some cases; and we deduce several consequences, including a classification of supercuspidal sheaves and  an orthogonal decomposition of the equivariant derived category of the nilpotent cone.
\end{abstract}

\maketitle

%%%%%%%%%%%%%%%%%%%%%%%%%%%%%%%%%%%%%%%%%%%%%%%%%%%%%%%%%%%%%%%%%%%%%%%%%%%
\section{Introduction}
\label{sec:intro}
%%%%%%%%%%%%%%%%%%%%%%%%%%%%%%%%%%%%%%%%%%%%%%%%%%%%%%%%%%%%%%%%%%%%%%%%%%%

%--------------------------------------------------------------------
\subsection{Modular generalized Springer correspondence}
%--------------------------------------------------------------------

Let $G$ be a connected complex reductive group. Let $Z(G)$ denote its centre, and let $Z(G)^\circ \subset Z(G)$ be the identity component.  A prime number $\ell$ is said to be \emph{rather good} for $G$ if it is good for $G$ and does not divide the order of the finite group $Z(G)/Z(G)^\circ$.

In this paper, which follows the series~\cite{genspring1, genspring2, genspring3} on the modular generalized Springer correspondence, we consider $G$-equivariant constructible sheaves on the nilpotent cone $\cN_G$ of $G$ with coefficients in a field $\bk$ of rather good characteristic. We will study not only the category of $G$-equivariant perverse sheaves $\Perv_G(\cN_G,\bk)$, as in~\cite{genspring1, genspring2, genspring3}, but also the $G$-equivariant bounded derived category $\Db_G(\cN_G,\bk)$. The main idea is that in rather good characteristic, it is possible to make meaningful comparisons with characteristic-$0$ generalized Springer theory, a subject that has been extensively studied for over thirty years since the work of Lusztig~\cite{lusztig}.  (In contrast, many of the statements in this paper are meaningless or false when the characteristic is not rather good.)

Let us briefly review the statement of the modular generalized Springer correspondence.  Let $\bk$ be a field of any characteristic that is `big enough' in the sense of~\eqref{eqn:definitely-big-enough} below. (Algebraically closed fields are always big enough.)  Let $\fN_{G,\bk}$ be the set of pairs $(\cO,\cE)$ where $\cO \subset \cN_G$ is a nilpotent orbit and $\cE$ is an irreducible $G$-equivariant $\bk$-local system on $\cO$ (up to isomorphism). Recall that $(\cO,\cE)$ is said to be \emph{cuspidal} if the corresponding simple perverse sheaf $\IC(\cO,\cE)$ does not occur as a quotient of any perverse sheaf induced from a proper Levi subgroup. According to~\cite[\mgmainthm]{genspring3}, if $\fL$ is a system of representatives for the $G$-conjugacy classes of Levi subgroups of $G$, there is a partition
\begin{equation}\label{eqn:genspring}
\fN_{G,\bk} = \bigsqcup_{\substack{\text{$L \in \fL$}\\ \text{$(\cO_L,\cE_L) \in \fN_{L,\bk}$ a cuspidal pair}}}
\fN_{G,\bk}^{(L,\cO_L,\cE_L)}
\end{equation}
into subsets called \emph{induction series}.  Moreover, for each induction series, there is a natural bijection
\begin{equation}\label{eqn:genspring-series}
\fN_{G,\bk}^{(L,\cO_L,\cE_L)} \overset{\sim}{\longleftrightarrow} \Irr(\bk[N_G(L)/L]).
\end{equation}

%--------------------------------------------------------------------
\subsection{$\ell$-cuspidal pairs and induction $\ell$-series}
%--------------------------------------------------------------------

The observation that gets us started is that when the characteristic $\ell$ of $\bk$ is rather good for $G$, there is a canonical bijection (see \S\ref{ss:0cusp-0series})
\begin{equation}\label{eqn:fn-indep}
\fN_{G,\bk} \overset{\sim}{\longleftrightarrow} \fN_{G,\C}.
\end{equation} 
Thus, the left-hand side of~\eqref{eqn:genspring} is `independent of $\bk$'.
In the introduction, for brevity, we will identify the two sides of~\eqref{eqn:fn-indep}, and simply write $\fN_G$. The set of cuspidal pairs only depends on $\ell$; therefore we call the corresponding elements of $\fN_G$ \emph{$\ell$-cuspidal} pairs. Similarly, the partition of $\fN_G$ given by~\eqref{eqn:genspring} depends only on $\ell$~\cite[\mgindepk]{genspring3}; we call the subsets on the right-hand side of~\eqref{eqn:genspring} \emph{induction $\ell$-series}.

A natural question is: how are the induction $\ell$-series related to the induction $0$-series?  The first half of the paper is devoted to the proof of the following result.

\begin{thm}\label{thm:main-intro}
Let $\ell$ be a rather good prime for $G$.  The partition of $\fN_G$ into induction $\ell$-series is a refinement of the partition into induction $0$-series.  In other words, each induction $0$-series is a union of induction $\ell$-series.
\end{thm}

The proof of Theorem~\ref{thm:main-intro} uses some observations obtained from the explicit classification of cuspidal pairs, but no explicit description of the partition~\eqref{eqn:genspring} (which is unknown in many cases).

%--------------------------------------------------------------------
\subsection{Cleanness conjecture}
%--------------------------------------------------------------------

The second half of the paper is devoted to the study of \emph{cleanness}.  A pair $(\cO,\cE) \in \fN_G$ is called \emph{$\ell$-clean} if the corresponding simple perverse sheaf $\IC(\cO,\cE)$ %with coefficients in a large enough field of characteristic $\ell$ 
has vanishing stalks on $\overline{\cO} \smallsetminus \cO$.  An important feature of characteristic-$0$ generalized Springer theory is that all cuspidal pairs are $0$-clean. (This was first proved by Lusztig~\cite{charsh5}; see also~\cite{rr}.) 
%We expect this to generalize in the following way.
The following is part of a series of (unpublished) conjectures by C.~Mautner.

\begin{conj}[Mautner's cleanness conjecture]\label{conj:clean-intro}
Let $\ell$ be a rather good prime for $G$. Then every $0$-cuspidal pair $(\cO,\cE)  \in \fN_G$ is $\ell$-clean.
\end{conj}

%This conjecture is closely related to a conjecture of C.~Mautner, asserting that every $\ell$-supercuspidal pair is clean. Here, a pair $(\cO,\cE) \in \fN_G$ is said to be \emph{$\ell$-supercuspidal} if the corresponding simple perverse sheaf does not occur as a composition factor of any induced sheaf in characteristic $\ell$.  (In characteristic~$0$, the induction functor takes semisimple perverse sheaves to semisimple perverse sheaves, so a pair is $0$-supercuspidal if and only if it is $0$-cuspidal; but when $\ell>0$ the two notions differ.) %With this terminology, Mautner's conjecture asserts that every $\ell$-supercuspidal pair is clean.

So far, we have been able to prove the cleanness conjecture in the following cases. (Here, $W$ denotes the Weyl group of $G$.) 

\begin{thm}\label{thm:clean-intro}
Conjecture~{\rm \ref{conj:clean-intro}} holds if $\ell \nmid |W|$, or if every irreducible factor of the root system of $G$ is either of type $A$, of type $B_4$, of type $C_3$, of type $D_5$, or of exceptional type.
\end{thm}

\noindent
The ranks in types $B$/$C$/$D$ mentioned here are the smallest for which there exist $0$-cuspidal pairs (see Remark~\ref{rmk:bcd-cuspidal}); of course, the conjecture is vacuously true if there are no $0$-cuspidal pairs for $G$.
%Here we have not listed cases where the conjecture is vacuously true because there are no $0$-cuspidal pairs, such as when the root system of $G$ has an irreducible factor of type $C_n$ where $n$ is not a triangular number, or of type $B_{(m-1)/2}$ or $D_{m/2}$ where $m$ is neither a triangular number nor a square (see~\cite[Introduction]{lusztig}).

\begin{rmk}
Conjecture~\ref{conj:clean-intro} was also independently checked by Mautner in type $A$ and in types $B_4$ and $C_3$ (using arguments different from ours, although also based on some explicit computations).
\end{rmk}

\begin{rmk}
Conjecture~\ref{conj:clean-intro} would be false in general if we replaced `$0$-cuspidal pair' with `$\ell$-cuspidal pair'. For example, when $G=\GL(2)$ and $\ell=2$ (which is rather good for $\GL(2)$), the unique $2$-cuspidal pair is not $2$-clean; see~\cite[Remark 2.5]{genspring1}. This $2$-cuspidal pair is not $0$-cuspidal for $\GL(2)$. (It is the modular reduction of the unique characteristic-$0$ cuspidal pair for $\SL(2)$, but note that $2$ is not rather good for $\SL(2)$.) However, Theorem~\ref{thm:clean-intro} and Lemma~\ref{lem:cogood-reduction} below reduce the proof of Conjecture~\ref{conj:clean-intro} to the case where $G$ is quasi-simple and simply connected of type $B$, $C$, or $D$. For such $G$, $\ell$ is rather good for $G$ if and only if $\ell>2$, and in this case the $\ell$-cuspidal pairs and the $0$-cuspidal pairs coincide by~\cite[Theorems 7.2, 8.3 and 8.4]{genspring2}. 
\end{rmk}

%--------------------------------------------------------------------
\subsection{Consequences}
%--------------------------------------------------------------------

The cleanness conjecture has a number of consequences for the structure of the equivariant derived category $\Db_G(\cN_G,\bk)$, recorded below.   
%In particular, by part~(1) below, our cleanness conjecture implies Mautner's conjecture in rather good characteristic. 
In this statement, a pair $(\cO,\cE) \in \fN_G$ is said to be \emph{$\ell$-supercuspidal} if the corresponding simple perverse sheaf does not occur as a composition factor of any induced sheaf in characteristic $\ell$.  (In characteristic~$0$, the induction functor takes semisimple perverse sheaves to semisimple perverse sheaves, so a pair is $0$-supercuspidal if and only if it is $0$-cuspidal; but when $\ell>0$ the two notions differ.) If $P \subset G$ is a parabolic subgroup of $G$ with Levi factor $L$ and if $(\cO_L, \cE_L)$ is an $\ell$-supercuspidal pair for $L$, the corresponding \emph{induction $\ell$-superseries} is the subset of $\fN_G$ consisting of pairs $(\cO',\cE')$ such that $\IC(\cO',\cE')$ is a composition factor of $\Ind_{L \subset P}^G \IC(\cO_L, \cE_L)$. (One can show that this subset does not depend on the choice of $P$; see \S\ref{ss:supercuspidal} below.)

\begin{thm}\label{thm:conseq-intro}
Let $\ell$ be a rather good prime for $G$, and assume that the cleanness conjecture holds for all Levi subgroups of $G$.  Then
\begin{enumerate}
\item A pair $(\cO,\cE) \in \fN_G$ is $\ell$-supercuspidal if and only if it is $0$-cuspidal.
\item The induction $\ell$-superseries coincide with the induction $0$-series.  In particular, the induction $\ell$-superseries are disjoint and give a partition of $\fN_G$.
\item 
\label{it:decom-DG}
%Let $\bk$ be a large enough field of characteristic $\ell$.  -- A standing assumption which has already been implicitly used. AH
There is an orthogonal decomposition
\[
\Db_G(\cN_G,\bk) \cong
\bigoplus_{\substack{\text{\normalfont $L \in \fL$} \\ \text{\normalfont $(\cO_L,\cE_L) \in \fN_L$ an $\ell$-supercuspidal pair}}} \Db_G(\cN_G,\bk)^{\super(L,\cO_L,\cE_L)},
\]
where $\Db_G(\cN_G,\bk)^{\super(L,\cO_L,\cE_L)}$ is the full triangulated subcategory of $\Db_G(\cN_G,\bk)$ generated by simple perverse sheaves in the induction $\ell$-superseries associated to $(L,\cO_L,\cE_L)$.
\end{enumerate}
\end{thm}

\begin{rmk}
Mautner also proved that
a decomposition similar to that in Theorem~\ref{thm:conseq-intro}\eqref{it:decom-DG} follows from Conjecture~\ref{conj:clean-intro} together with some other conjectures of his that we do not consider in this paper.
\end{rmk}

The decomposition of $\Db_G(\cN_G,\bk)$ in the last part of this theorem implies a corresponding decomposition of the abelian category $\Perv_G(\cN_G,\bk)$.  (In fact, the proof is structured in such a way that we obtain the abelian category version first.) The analogous decomposition of $\Db_G(\cN_G,\C)$ was proved by Rider and Russell~\cite[Theorem 3.5]{rr} following arguments of Lusztig.

\begin{rmk}
Thanks to Theorem~\ref{thm:clean-intro}, the assumption in the statement of Theorem~\ref{thm:conseq-intro} certainly holds if $\ell \nmid |W|$, or if every irreducible factor of the root system of $G$ is either of type $A$, of type $B_n$ for $n<7$, of type $C_n$ for $n<6$, of type $D_n$ for $n<8$, or of exceptional type (see Remark~\ref{rmk:bcd-cuspidal}).
\end{rmk}

%--------------------------------------------------------------------
\subsection{Organization of the paper}
%--------------------------------------------------------------------

The paper is organized as follows.  Section~\ref{sec:notation} contains a review of relevant terminology, notation, and results from~\cite{genspring1,genspring2,genspring3}, along with a few easy lemmas.  Section~\ref{sec:partialorder} concerns a certain partial order on the set of cuspidal data, defined using the Lusztig stratification of the Lie algebra $\fg$.  Theorem~\ref{thm:main-intro} is proved in Section~\ref{sec:comparison}.

The remainder of the paper is devoted to cleanness.  Section~\ref{sec:clean} introduces the cleanness conjecture and proves all but two cases of Theorem~\ref{thm:clean-intro}.  The remaining cases, those of a group of type $E_8$ with $\ell = 7$ and of a group of type $B_4$ with $\ell=3$, are handled in Section~\ref{sec:e8-clean}.  Finally, Section~\ref{sec:consequences} contains the proof of Theorem~\ref{thm:conseq-intro}, along with a number of ancillary results.

%--------------------------------------------------------------------
\subsection{Acknowledgements}
%--------------------------------------------------------------------

We thank Carl Mautner for discussions concerning his conjectures, which motivated some of the results in the second half of the paper.
We used the development version  \cite{chevie-dev} of the GAP Chevie package \cite{chevie}.

%%%%%%%%%%%%%%%%%%%%%%%%%%%%%%%%%%%%%%%%%%%%%%%%%%%%%%%%%%%%%%%%%%%%%%%%%%%
\section{Conventions and preliminaries}
\label{sec:notation}
%%%%%%%%%%%%%%%%%%%%%%%%%%%%%%%%%%%%%%%%%%%%%%%%%%%%%%%%%%%%%%%%%%%%%%%%%%%

%--------------------------------------------------------------------------
\subsection{Coefficients}
%--------------------------------------------------------------------------

We retain the meanings of $G$, $\cN_G$, and $\fN_{G,\bk}$ from the introduction. (However, the notation `$\fN_G$' with the coefficients suppressed will not be used again.)   Let $\Db_G(\cN_G,\bk)$ be the $G$-equivariant bounded derived category of $\cN_G$ with coefficients in $\bk$, and let $\Perv_G(\cN_G,\bk) \subset \Db_G(\cN_G,\bk)$ be the abelian category of $G$-equivariant perverse sheaves on $\cN_G$.  The set of isomorphism classes of simple objects in $\Perv_G(\cN_G,\bk)$ is naturally in bijection with $\fN_{G,\bk}$.

Let $\ell$ denote the characteristic of $\bk$.  Throughout the paper (unless otherwise specified), we assume that $\ell$ is a rather good prime for $G$.  As in~\cite{genspring3}, we also assume that $\bk$ satisfies the following condition:
\begin{equation} \label{eqn:definitely-big-enough}
\begin{array}{c}
\text{for every Levi subgroup $L$ of $G$ and pair $(\cO_L,\cE_L)\in\fN_{L,\bk}$,}\\
\text{the irreducible $L$-equivariant local system $\cE_L$ is absolutely irreducible.}
\end{array}
\end{equation}
See~\cite[\mgbigenough]{genspring3} for a more explicit translation of this condition.
Furthermore, we assume throughout that there exists a finite extension $\K$ of $\Ql$ that satisfies~\eqref{eqn:definitely-big-enough}, and such that $\bk$ is the residue field of the ring $\O$ of integers of $\K$.  Thus, the triple $(\K,\O,\bk)$ constitutes an $\ell$-modular system.  With this set-up, we can invoke the machinery of \emph{modular reduction} from~\cite[\S 2.7]{genspring1}.  

The lemmas below, which are restatements of~\cite[\mgrathergood]{genspring3}, collect some basic facts about rather good primes.  Here, $A_G(x)$ denotes the component group of the stabilizer of a point $x \in \cN_G$.

\begin{lem}\label{lem:rgood-defn}
The following conditions on a prime $\ell$ are equivalent:
\begin{enumerate}
\item $\ell$ does not divide $|A_G(x)|$ for any $x \in \cN_G$.\label{it:rgood-ag}
\item $\ell$ is good for $G$ and does not divide $|Z(G)/Z(G)^\circ|$.\label{it:rgood-zg}
\end{enumerate}
\end{lem}

(In~\cite{genspring3}, part~\eqref{it:rgood-ag} was taken to be the definition of \emph{rather good}, rather than part~\eqref{it:rgood-zg}.)

\begin{lem}\label{lem:rgood-basic}
Let $\ell$ be a prime.
\begin{enumerate}
\item If $\ell$ is rather good for $G$, it is rather good for every Levi subgroup of $G$.\label{it:rgood-levi}
\item If $\ell$ does not divide $|W|$, then it is rather good for $G$.\label{it:rgood-easy}
\item If $G$ is quasi-simple and not of type $A$, then $\ell$ is rather good for $G$ if and only if it is good for $G$.\label{it:rgood-simple}
\end{enumerate}
\end{lem}

%--------------------------------------------------------------------------
\subsection{Cuspidal and supercuspidal data}
\label{ss:supercuspidal}
%--------------------------------------------------------------------------

Recall that for a parabolic subgroup $P\subset G$ with Levi factor $L$, we have an induction functor~\cite[\S 2.1]{genspring1}
\[
\Ind_{L \subset P}^G : \Perv_{L}(\cN_{L},\bk) \to \Perv_G(\cN_G,\bk).
\]
A pair $(\cO,\cE) \in \fN_{G,\bk}$ is said to be \emph{cuspidal} (resp.~\emph{supercuspidal}) if the simple perverse sheaf $\IC(\cO,\cE)$ does not occur as a quotient (resp.~composition factor) of any induced perverse sheaf $\Ind_{L \subset P}^G(\cF)$ with $L \subsetneq G$.  The set of cuspidal, resp. supercuspidal, pairs in $\fN_{G,\bk}$ is denoted
\[
\fN^\cusp_{G,\bk}, 
\qquad\text{resp.}\qquad
\fN^\supercusp_{G,\bk}.
\]
Obviously, every supercuspidal pair is cuspidal.

A \emph{cuspidal datum} (resp.~\emph{supercuspidal datum}) is a triple $(L,\cO_L,\cE_L)$, where $L \subset G$ is a Levi subgroup, and $(\cO_L,\cE_L)$ belongs to $\fN^\cusp_{L,\bk}$ (resp.~$\fN^\supercusp_{L,\bk}$).  Let $\fM_{G,\bk}$ denote the set of $G$-orbits of cuspidal data.  Let $\fM^\sup_{G,\bk} \subset \fM_{G,\bk}$ be the subset consisting of $G$-orbits of supercuspidal data.  In a minor abuse of language, we will often call elements of $\fM_{G,\bk}$ `cuspidal data', omitting any mention of the $G$-action.

Recall that the \emph{induction series} associated to a cuspidal datum $(L,\cO_L,\cE_L)$ is the set
\[
\fN_{G,\bk}^{(L,\cO_L,\cE_L)} = \{ (\cO',\cE') \in \fN_{G,\bk} \mid
\text{$\IC(\cO',\cE')$ is a quotient of $\Ind_{L \subset P}^G \IC(\cO_L,\cE_L)$} \}.
\]
Here $P$ is a parabolic subgroup of $G$ with Levi factor $L$; the set $\fN_{G,\bk}^{(L,\cO_L,\cE_L)}$ does not depend on the choice of $P$ (see~\cite[\S 2.2]{genspring2}), and depends only on the $G$-conjugacy class of $(L,\cO_L, \cE_L)$.

Similarly, if $(L,\cO_L,\cE_L)$ is a supercuspidal datum, we define the corresponding \emph{induction superseries} to be the set
\[
\fN_{G,\bk}^{\super (L,\cO_L,\cE_L)} = \{ (\cO',\cE') \in \fN_{G,\bk} \mid
\text{$\IC(\cO',\cE')$ is a composition factor of $\Ind_{L \subset P}^G \IC(\cO_L,\cE_L)$} \}.
\]
Here again $P$ is a parabolic subgroup of $G$ with Levi factor $L$;
it follows from Lemma~\ref{lem:comparison-cusp} below (and~\cite[\S 2.2]{genspring2} again) that the set $\fN_{G,\bk}^{\super (L,\cO_L,\cE_L)}$ does not depend on the choice of $P$, and depends only on the $G$-conjugacy class of $(L,\cO_L, \cE_L)$.

With characteristic-$0$ coefficients, the Decomposition Theorem implies that the induction fuctor $\Ind_{L \subset P}^G$ takes semisimple perverse sheaves to semisimple perverse sheaves, so in that case, the notions of `cuspidal' and `supercuspidal' coincide, as do the notions of `induction series' and `induction superseries.' These properties no longer hold in positive characteristic; see~\cite[Remark~3.2]{genspring1}.

%--------------------------------------------------------------------------
\subsection{0-cuspidal pairs and 0-series}
\label{ss:0cusp-0series}
%--------------------------------------------------------------------------

By Lemma~\ref{lem:rgood-defn}\eqref{it:rgood-ag}, if $\ell$ is rather good, then there is a canonical bijection $\Irr(\bk[A_G(x)]) \cong \Irr(\K[A_G(x)])$ for each $x \in \cN_G$, and hence a canonical bijection
\begin{equation}\label{eqn:fn-bij}
\theta_G: \fN_{G,\K} \simto \fN_{G,\bk},
\end{equation}
of which we saw an incarnation in~\eqref{eqn:fn-indep}.  (Although it was convenient in the introduction to identify the two sides of~\eqref{eqn:fn-indep}, the proofs in this paper are clearer when we retain the distinction.)  Furthermore, because $\K$ satisfies~\eqref{eqn:definitely-big-enough}, Lusztig's results~\cite{lusztig} on the generalized Springer correspondence with coefficients in $\Qlb$ hold over $\K$ as well.  We will freely make use of this observation throughout the paper.

By Lemma~\ref{lem:rgood-basic}\eqref{it:rgood-levi}, there are also similar bijections
\[
\theta_L : \fN_{L,\K} \simto \fN_{L,\bk}
\]
for any Levi subgroup $L \subset G$.

A pair $(\cO,\cE) \in \fN_{G,\bk}$ is \emph{$0$-cuspidal} if it lies in $\theta_G(\fN^\cusp_{G,\K})$.  The set of $0$-cuspidal pairs is denoted by
\[
\fN^\zcusp_{G,\bk} = \theta_G(\fN^\cusp_{G,\K}).
\]
A \emph{$0$-cuspidal datum} is a triple $(L,\cO_L,\cE_L)$ where $L \subset G$ is a Levi subgroup, and $(\cO_L,\cE_L)$ belongs to $\fN_{L,\bk}^\zcusp$.  Let $\fM_{G,\bk}^0$ denote the set of $G$-orbits of $0$-cuspidal data.  The \emph{induction $0$-series} associated to a $0$-cuspidal datum $(L,\cO_L,\cE_L)$ is the subset $\fN_{G,\bk}^{\zs(L,\cO_L,\cE_L)} \subset \fN_{G,\bk}$ given by
\[
\fN^{\zs(L,\cO_L,\cE_L)}_{G,\bk} = \theta_G \Bigl( \fN^{(L,\theta_L^{-1}(\cO_L,\cE_L))}_{G,\K} \Bigr).
\]
In other words, an induction $0$-series is simply the image under~\eqref{eqn:fn-bij} of a character\-istic-$0$ induction series. 

Finally, the notion of $0$-cuspidal pairs gives rise to another kind of series, but in $\fM_{G,\bk}$ rather than in $\fN_{G,\bk}$.  Given $(L,\cO_L,\cE_L) \in \fM_{G,\bk}^0$, let
\begin{equation}\label{eqn:data-0-series}
\fM_{G,\bk}^{\zs(L,\cO_L,\cE_L)} = \{ (M, \cO_M, \cE_M) \mid \text{
$L \subset M \subset G$ and $(\cO_M,\cE_M) \in \fN_{M,\bk}^\cusp \cap \fN_{M,\bk}^{\zs(L,\cO_L,\cE_L)}$} \}/(\text{$G$-conjugacy}).
\end{equation}
A set of this form is called a \emph{$0$-series of cuspidal data}.  It is manifestly a subset of $\fM_{G,\bk}$.

\begin{lem}
\label{lem:comparison-cusp}
We have
$
\fN^\supercusp_{G,\bk} \subset \fN^\zcusp_{G,\bk} \subset \fN^{\cusp}_{G,\bk}
$.
\end{lem}

This lemma implies that we have natural embeddings $\fM_{G,\bk}^\sup \subset \fM_{G,\bk}^0 \subset \fM_{G,\bk}$.

\begin{proof}
We begin with the second inclusion. Let $(\cO,\cE) \in \fN^\zcusp_{G,\bk}$, and let $\cE^{\K}$ be the equivariant $\K$-local system on $\cO$ such that $\theta_G(\cO,\cE^\K)=(\cO, \cE)$. Then $\IC(\cO,\cE)$ occurs in the modular reduction of $\IC(\cO,\cE^\K)$, so it is cuspidal by~\cite[Proposition~2.22]{genspring1}.

On the other hand,
suppose that $(\cO,\cE) \in \fN_{G,\bk}$ is not $0$-cuspidal.  As above, let $\cE^{\K}$ be the equivariant $\K$-local system on $\cO$ such that $\theta_G(\cO,\cE^\K)=(\cO, \cE)$, so that $\IC(\cO,\cE)$ occurs in the modular reduction of $\IC(\cO,\cE^\K)$.  Since $(\cO,\cE)$ is not $0$-cuspidal, $\IC(\cO,\cE^\K)$ occurs as a direct summand of some $\Ind_{L \subset P}^G (\IC(\cO_L, \cE^\K_L))$ with $L \neq G$.  Since induction commutes with modular reduction (see~\cite[Remark~2.23]{genspring1}), $\IC(\cO,\cE)$ occurs as a composition factor in the perverse sheaf $\Ind_{L \subset P}^G (\bk \lotimes_\O \IC(\cO_L, \cE^\O_L))$, where $\cE^\O_L$ is any $\O$-form of $\cE^\K_L$.  In particular, $(\cO,\cE)$ is not supercuspidal.
\end{proof}

In~\cite[\mgthmcusp]{genspring3} we classified cuspidal pairs in good characteristic. As an immediate consequence of this classification, we have the following result in our current setting of rather good characteristic:

\begin{prop} \label{prop:simply-conn}
If $G$ is semisimple and simply connected, then every cuspidal pair is $0$-cuspidal, so $\fN^\zcusp_{G,\bk}=\fN^{\cusp}_{G,\bk}$ in this case.
\end{prop}

\noindent
Of course, the property of being semisimple and simply connected is not inherited by Levi subgroups, so we cannot conclude that $\fM_{G,\bk}^0=\fM_{G,\bk}$ for such $G$ (and indeed this is false in general). 

%--------------------------------------------------------------------------
\subsection{A `reduction lemma'}
%--------------------------------------------------------------------------

The following lemma will be used below to reduce the proof of some statements to the case where $G$ is semisimple of type $A$, or simply connected and quasi-simple not of type $A$. (Here and throughout this section, a `semisimple group of type $A$' means a semisimple group whose root system is a product of root systems of type $A$.)

\begin{lem}
\label{lem:cogood-reduction}
Given a reductive group $G$, there exists a semisimple group $G'$ with the following properties:
\begin{enumerate}
\item There is an isogeny $G' \twoheadrightarrow G/Z(G)^\circ$.  Thus, $\cN_G$ can be identified with $\cN_{G'}$, and there is a fully faithful functor $\Perv_G(\cN_G,\bk) \to \Perv_{G'}(\cN_{G'},\bk)$.
\item We have $G' \cong G'_1 \times G'_2$, where $G'_1$ is a semisimple group of type $A$, and $G'_2$ is a product of quasi-simple, simply connected groups that are \emph{not} of type $A$.
\item A prime number $\ell$ is rather good for $G$ if and only if it is rather good for $G'$. If it is rather good for both, we have
\[
\Perv_{G'}(\cN_{G'},\bk) \cong \Perv_G(\cN_G,\bk) \oplus \Perv_G(\cN_G,\bk)^\perp,
\]
where $\Perv_G(\cN_G,\bk)^\perp$ denotes the full subcategory of $\Perv_{G'}(\cN_{G'},\bk)$ consisting of objects with no quotient or subobject in $\Perv_G(\cN_G,\bk)$.\label{it:cogood-reduc-sum}
\end{enumerate}
\end{lem}

\begin{proof}
First we claim that, for any field $\bk$, the natural functor $\Perv_{G/Z(G)^\circ}(\cN_G,\bk) \to \Perv_G(\cN_G,\bk)$ is an equivalence of categories. Indeed the category $\Perv_{G/Z(G)^\circ}(\cN_G,\bk)$, resp.~$\Perv_G(\cN_G,\bk)$, is equivalent to the full subcategory of the category $\Perv(\cN_G,\bk)$ whose objects are the perverse sheaves $\cF$ such that the pullbacks of $\cF$ to $G/Z(G)^\circ \times \cN_G$, resp.~to $G \times \cN_G$, under the morphisms given the action and the projection are isomorphic. Now, since the projection $G \times \cN_G \to G/Z(G)^\circ \times \cN_G$ is smooth with connected fibres, the (shifted) pullback functor $\Perv(G/Z(G)^\circ \times \cN_G) \to \Perv(G \times \cN_G)$ is fully faithful (see~\cite[Proposition~4.2.5]{bbd}), so that an object of $\Perv(\cN_G,\bk)$ belongs to $\Perv_{G/Z(G)^\circ}(\cN_G,\bk)$ iff it belongs to $\Perv_{G}(\cN_G,\bk)$.
Using this claim, we can (and will) assume that $G$ is semisimple.

Let $\tilde G$ be the universal cover of $G$, and let $K_0$ be the kernel of $\tilde G \twoheadrightarrow G$.  Define a subgroup $K \subset K_0$ as follows:
\begin{equation}\label{eqn:cogood-reduction}
K = 
\begin{cases}
K_0 & \text{if $\tilde G$ is a product of groups of type $A$,} \\
(K_0)_{2'} & \text{if $\tilde G$ contains factors of type $B$, $C$, or $D$,} \\
  & \quad\text{but not of exceptional type,} \\
(K_0)_{2',3'} & \text{if $\tilde G$ contains factors of exceptional type.}
\end{cases}
\end{equation}
Here, $(K_0)_{2'}$ (resp.~$(K_0)_{2',3'}$) denotes the subgroup of $K_0$ consisting of elements of order coprime to $2$ (resp.~coprime to $2$ and $3$).  Let $G' = \tilde G/K$.  Also, let $\tilde G_1$ (resp.~$G'_2$) be the product of all quasi-simple factors of $\tilde G$ of type $A$ (resp.~not of type $A$).  Thus, $\tilde G \cong \tilde G_1 \times G'_2$.

Now, the centre of a quasi-simple group of type $B$, $C$, or $D$ is a $2$-group, and the center of a quasi-simple exceptional group may have order $1$, $2$, or $3$.  Therefore, the subgroup $K \subset Z(\tilde G) = Z(\tilde G_1) \times Z(G'_2)$ must be of the form $K_1 \times \{1\}$ for some $K_1 \subset Z(\tilde G_1)$.
Then we have $G' \cong (\tilde G_1/K_1) \times G'_2$. Let $G'_1 = \tilde G_1/K_1$; this is a semisimple group of type $A$.  Since $G$ is a quotient of $G'$, there is a natural fully faithful functor $\Perv_G(\cN_G,\bk) \hookrightarrow \Perv_{G'}(\cN_{G'},\bk)$.

We now show that $G$ and $G'$ have the same set of rather good primes.  If $\tilde G = \tilde G_1$, then $G' = G$ and there is nothing to prove. If $\tilde G$ contains factors of type $B$, $C$, or $D$, but not of exceptional type, then $G$ and $G'$ have the same good primes, and $Z(G)$ is a quotient of $Z(G')$ by a $2$-group. Thus, $|Z(G)|$ and $|Z(G')|$ have the same odd prime divisors, so $G$ and $G'$ have the same rather good primes.
Similar reasoning applies when $\tilde G$ has exceptional factors.

Finally, assume that $\ell$ is rather good for $G$ and $G'$, and let $Z \cong K_0/K$ be the kernel of the projection $G' \twoheadrightarrow G$. Then by Lemma~\ref{lem:Perv-trivial-character} the fully faithful functor  $\Perv_G(\cN_G,\bk) \to \Perv_{G'}(\cN_{G'},\bk)$ identifies $\Perv_G(\cN_G,\bk)$ with the subcategory of $\Perv_{G'}(\cN_{G'},\bk)$ whose objects have trivial $Z$-character. Then the direct sum in part~\eqref{it:cogood-reduc-sum} comes from Lemma~\ref{lem:centralchar-decomp}; $\Perv_G(\cN_G,\bk)^\perp$ is the direct sum of all $\Perv_{G'}(\cN_G,\bk)_\chi$ where $\chi$ is \emph{not} trivial on $K$.
\end{proof}

%------------------------------------------------------------------------
\subsection{Central characters}
\label{ss:central-char}
%------------------------------------------------------------------------

As explained in Appendix~\ref{app:central-char}, for any pair $(\cO,\cE) \in \fN_{G,\bk}$, the local system $\cE$ determines a central character $\chi: Z(G)/Z(G)^\circ \to \bk^\times$.  If $L$ is a Levi subgroup of $G$ and $(\cO_L, \cE_L) \in \fN_{L, \bk}$, then $\cE_L$ has a central character $Z(L)/Z(L)^\circ \to \bk^\times$.  In a slight abuse of language (following~\cite[\S 5.1]{genspring2}), we will also refer to the composition $Z(G)/Z(G)^\circ \twoheadrightarrow Z(L)/Z(L)^\circ \to \bk^\times$ as the `central character' of $\cE_L$.

%In this section, we establish a number of easy facts about (super)cuspidal pairs based on a study of central characters.

%--------------------------------------------------------------------------
%\subsection{Some consequences of the classification of cuspidal pairs}
%--------------------------------------------------------------------------

The following proposition is contained in~\cite[\mgthmcusp]{genspring3}.

\begin{prop}
\label{prop:unicity-cuspidal}
%\label{prop:0-cusp-not-A}
%Let $G$ be a quasi-simple, simply connected group.
% that is not of type $A$.
Any two distinct cuspidal pairs for $G$ have distinct central characters.
\end{prop}

%\begin{proof}
%For the classical groups, this follows from~\cite[Theorems~7.2, 8.3, and~8.4]{genspring2}; for the exceptional groups, see~\cite[\mgexcepcusp]{genspring3}. In all cases, the claim about central characters is a well-known observation of Lusztig about the $\Qlb$-generalized Springer correspondence.
%\end{proof}

%Now we come back to the case where $G$ is an arbitrary connected reductive group.
%
%\begin{prop}
%\label{prop:unicity-cuspidal}\footnote{Currently this is in the same situation as Proposition~\ref{prop:0-cusp-not-A}, in that it follows immediately from~\cite[\mgthmcusp]{genspring3}. So we don't need the proof. Maybe the two propositions should be stated together or at least adjacently? AH}
%For any central character $\chi$, there exists at most one cuspidal pair $(\cO,\cE)$ for $G$ with central character $\chi$.
%\end{prop}
%
%\begin{proof}
%By Lemma~\ref{lem:cogood-reduction}, we can assume that $G$ is either a semisimple group of type $A$, or a quasi-simple, simply connected group not of type $A$. In the first case the statement follows from~\cite[Theorem~6.3]{genspring2} (which does not require any assumption on $\ell$), and in the second case it is part of Proposition~\ref{prop:0-cusp-not-A}.
%\end{proof}

%--------------------------------------------------------------------------
%\subsection{Relation between induction $0$-series and superseries}
%--------------------------------------------------------------------------

Recall (see~\cite[Corollary~2.12]{genspring1}) that if $(\cO, \cE) \in \fN_{G,\bk}^{\cusp}$, then there exists a unique pair $(\cO', \cE') \in \fN_{G,\bk}^{\cusp}$ such that $\bT_{\fg}(\IC(\cO,\cE)) \cong \IC(\cO' + \fz_G, \cE' \boxtimes \underline{\bk})$, where $\bT_{\fg}$ is the Fourier--Sato transform and $\fz_G$ is the center of $\fg$. The map $\fN_{G,\bk}^{\cusp} \to \fN_{G,\bk}^{\cusp}$ sending $(\cO,\cE)$ to $(\cO', \cE')$ is an involution, which is often the identity. (In fact, we don't know any example where $(\cO', \cE') \neq (\cO,\cE)$.)
In particular this is the case under our assumption that $\ell$ is rather good.

\begin{cor}
\label{cor:involution=id}
For any $(\cO, \cE) \in \fN_{G,\bk}^{\cusp}$ we have $(\cO', \cE')=(\cO,\cE)$.
\end{cor}

\begin{proof}
As in~\cite[Corollary~6.6]{genspring2}, the claim is a consequence of Proposition~\ref{prop:unicity-cuspidal}.
\end{proof}

\begin{cor}
\label{cor:comparison-0series-superseries}
Let $(L,\cO_L,\cE_L)$ be a supercuspidal datum. Then we have
\[
\fN_{G,\bk}^{\zs(L,\cO_L,\cE_L)} \subset \fN^{\super(L,\cO_L,\cE_L)}_{G,\bk}.
\]
\end{cor}

\begin{proof}
First we note that $(L,\cO_L,\cE_L)$ is a $0$-cuspidal datum by Lemma~\ref{lem:comparison-cusp}, so that the notation $\fN_{G,\bk}^{\zs(L,\cO_L,\cE_L)}$ makes sense.  Let $\cE_L^{\K}$ be the equivariant $\K$-local system on $\cO_L$ such that $\theta_L(\cO_L,\cE_L^\K)=(\cO_L, \cE_L)$, and let $\cE_L^\O$ be an $\O$-form of $\cE_L^\K$. We claim that
\begin{equation}
\label{eqn:mod-reduction-0-cusp}
\bk \lotimes_\O \IC(\cO_L, \cE_L^\O) \cong \IC(\cO_L, \cE_L).
\end{equation}
Indeed, by~\cite[Proposition~2.22]{genspring1}, all the composition factors of the perverse sheaf $\bk \lotimes_\O \IC(\cO_L, \cE_L^\O)$ are cuspidal. Since they have the same central character as $\cE_L$, Proposition~\ref{prop:unicity-cuspidal} implies that the only possible composition factor is $\IC(\cO_L, \cE_L)$, proving~\eqref{eqn:mod-reduction-0-cusp}.

Now, let $(\cO,\cE) \in \fN_{G,\bk}^{\zs(L,\cO_L,\cE_L)}$, and let $\cE^{\K}$ be the equivariant $\K$-local system on $\cO$ such that $\theta_G(\cO,\cE^\K)=(\cO, \cE)$. By assumption, $\IC(\cO,\cE^\K)$ is a composition factor of $\Ind_{L \subset P}^G(\IC(\cO_L, \cE_L^\K))$ (where $P \subset G$ is any parabolic subgroup with Levi factor $L$). Since $\IC(\cO,\cE)$ appears in the modular reduction of $\IC(\cO,\cE^\K)$, this perverse sheaf is a composition factor of the perverse sheaf
\[
\bk \lotimes_\O \Ind_{L \subset P}^G(\IC(\cO_L, \cE_L^\O)) \cong \Ind_{L \subset P}^G(\bk \lotimes_\O \IC(\cO_L, \cE_L^\O)) \overset{\eqref{eqn:mod-reduction-0-cusp}}{\cong} \Ind_{L \subset P}^G (\IC(\cO_L, \cE_L)).
\]
This proves that $(\cO,\cE) \in \fN^{\super(L,\cO_L,\cE_L)}_{G,\bk}$, finishing the proof.
\end{proof}

Finally, we will need the following result about the type-$A$ case.

\begin{prop}\label{prop:0-series-A}
Let $G$ be a semisimple group of type $A$.
\begin{enumerate}
\item 
\label{it:0-cusp-char}
If $(L,\cO_L,\cE_L)$ and $(L',\cO_{L'},\cE_{L'})$ are $0$-cuspidal data which are not $G$-conjugate, then $\cE_L$ and $\cE_{L'}$ have distinct central characters.
\item 
\label{it:0-ser-char}
A pair $(\cO,\cE) \in \fN_{G,\bk}$ lies in $\fN^{\zs(L,\cO,\cE_L)}_{G,\bk}$ if and only if $\cE$ has the same central character as $\cE_L$.
\item
\label{it:0-ser-sum}
We have
\[
\Perv_G(\cN_G,\bk) \cong \bigoplus_{(L,\cO_L,\cE_L) \in \fM_{G,\bk}^0}
\Serre(\fN^{\zs(L,\cO_L,\cE_L)}_{G,\bk}),
\]
where $\Serre(\fN^{\zs(L,\cO_L,\cE_L)}_{G,\bk})$ denotes the Serre subcategory of $\Perv_G(\cN_G,\bk)$ generated by the simple objects $\IC(\cO,\cE)$ with $(\cO,\cE) \in \fN^{\zs(L,\cO_L,\cE_L)}_{G,\bk}$.
\end{enumerate}
\end{prop}

\begin{proof}
By our assumptions on $\ell$ and $\bk$, the $\bk$-characters of $Z(G)$ are in natural bijection with the $\Qlb$-characters of $Z(G)$. Therefore, statements~\eqref{it:0-cusp-char} and~\eqref{it:0-ser-char} are equivalent to their counterpart for the $\Qlb$-generalized Springer correspondence. These counterparts are well known; see~\cite[\S 10.3]{lusztig}. The decomposition in~\eqref{it:0-ser-sum} then
follows from Lemmas~\ref{lem:central-char-Perv} and~\ref{lem:centralchar-decomp}.
\end{proof}

%%%%%%%%%%%%%%%%%%%%%%%%%%%%%%%%%%%%%%%%%%%%%%%%%%%%%%%%%%%%%%%%%%%%%%%%%%%
\section{The partial order on cuspidal data}
\label{sec:partialorder}
%%%%%%%%%%%%%%%%%%%%%%%%%%%%%%%%%%%%%%%%%%%%%%%%%%%%%%%%%%%%%%%%%%%%%%%%%%%

%--------------------------------------------------------------------------
\subsection{Definition of the order}
%--------------------------------------------------------------------------

There is a natural partial order on the set $\fM_{G,\bk}$, defined as follows:
\[
(L,\cO_L,\cE_L) \preceq_G (M,\cO_{M}, \cE_{M})
\qquad\text{if}\qquad
\begin{array}{c}
\text{$\cE_L$ and $\cE_{M}$ have the same central} \\
\text{character, and $Y_{(M,\cO_{M})} \subset \overline{Y_{(L,\cO_L)}}$.}
\end{array}
\]
Here $Y_{(L,\cO_L)}$ is the Lusztig stratum associated to $(L,\cO_L)$; see~\cite[\S 2.6]{genspring1} or~\cite[\S 2.1]{genspring2}. (When necessary, below we will add a superscript `$G$' to the notation.) Note that if $(L,\cO_L,\cE_L) \preceq_G (M,\cO_{M}, \cE_{M})$ and $(M,\cO_M,\cE_M) \preceq_G (L,\cO_{L}, \cE_{L})$, then
$Y_{(M,\cO_{M})} = Y_{(L,\cO_L)}$, so that $M$ and $L$ are $G$-conjugate. Since, under our assumptions, these groups have at most one cuspidal pair with a given central character (see Proposition~\ref{prop:unicity-cuspidal}), this implies that $(L,\cO_L,\cE_L)$ and $(M,\cO_{M}, \cE_{M})$ are $G$-conjugate, so that $\preceq_G$ is indeed an order.  The following alternative description of this order is due to Lusztig. (Note that $\preceq$ has a different meaning in~\cite{lusztig-cusp2}: it is a refinement of the opposite of the partial order denoted $\preceq_G$ here.) 

\begin{prop}[{\cite[Proposition~6.5]{lusztig-cusp2}}]
\label{prop:lusztig-po}
Let $(L,\cO_L,\cE_L)$ and $(M,\cO_M,\cE_M)$ be two cuspidal data.  We have $(L,\cO_L,\cE_L) \preceq_G (M,\cO_{M}, \cE_{M})$ if and only if the following three conditions all hold:
\begin{enumerate}
\item $\cE_L$ and $\cE_M$ have the same central character.
\item There exists an element $g \in G$ such that $g L g^{-1} \subset M$.
\item The orbit $\cO_M$ is contained in the closure of the induced orbit $\mathrm{Ind}_{g L g^{-1}}^M (g \cdot \cO_L)$.
\end{enumerate}
\end{prop}

An immediate consequence is that if $M \subset G$ is a Levi subgroup that contains both $L$ and $L'$, then 
\begin{equation}
\label{eqn:order-Levi}
(L,\cO_L,\cE_L) \preceq_M (L',\cO_{L'}, \cE_{L'})
\qquad\text{implies}\qquad
(L,\cO_L,\cE_L) \preceq_G (L',\cO_{L'}, \cE_{L'}).
\end{equation}
(This can also be deduced from the description of $\overline{Y_{(L,\cO_L)}}$ recalled in~\cite[\S 2.6]{genspring1}.)  The opposite implication is certainly false in general; for example, it can happen that $L$ and $L'$ are $G$-conjugate but not $M$-conjugate.  Nevertheless, we will see a partial converse in Corollary~\ref{cor:order-0-Levi} below.

%--------------------------------------------------------------------------
\subsection{Induction series and the order $\preceq_G$}
%--------------------------------------------------------------------------

\begin{lem}
\label{lem:0-series-ineq}
If $(\cO,\cE) \in \fN_{G,\bk}$ lies in the induction series $\fN^{(M,\cO_M, \cE_M)}_{G,\bk}$ and in the induction $0$-series $\fN^{\zs(L,\cO_L,\cE_L)}_{G,\bk}$, then $(L,\cO_L,\cE_L) \preceq_G (M, \cO_M, \cE_M)$.
\end{lem}

\begin{proof}
Let $\cE^\K$ be the equivariant $\K$-local system on $\cO$ such that $\theta_G(\cO,\cE^\K)=(\cO, \cE)$, and let $\cE^\O$ be an $\O$-form of $\cE^\K$.  Then we have
\begin{equation}
\label{eqn:isom-0-series-ineq}
\K \otimes_\O \bT_\fg^\O(\IC(\cO,\cE^\O)) \cong \bT_\fg^\K(\IC(\cO,\cE^\K)) \cong \IC(Y_{(L,\cO_L)}, \cD^\K)
\end{equation}
for some local system $\cD^\K$ on $Y_{(L,\cO_L)}$, where $\bT^\K_\fg$, resp.~$\bT^\O_\fg$, is the Fourier--Sato transform on $\fg$ with coefficients $\K$, resp.~$\O$. Indeed, the first isomorphism follows from the compatibility of Fourier--Sato transform with extension of scalars (and the fact that $\K \otimes_\O \IC(\cO,\cE^\O) \cong \IC(\cO,\cE^\K)$), and the second isomorphism follows from~\cite[Lemma~2.1]{genspring2} and Corollary~\ref{cor:involution=id}.

From~\eqref{eqn:isom-0-series-ineq} we deduce, using~\cite[Proposition~2.8]{juteau}, that the perverse sheaf $\bT_\fg^\O(\IC(\cO,\cE^\O))$ is supported on $\overline{Y_{(L,\cO_L)}}$. Therefore, the same property holds for
\[
\bk \lotimes_\O \bT_\fg^\O(\IC(\cO,\cE^\O)) \cong \bT_\fg^\bk(\bk \lotimes_\O \IC(\cO, \cE^\O)).
\]
Now $\IC(\cO,\cE)$ is a composition factor of $\bk \lotimes_\O \IC(\cO, \cE^\O)$, and hence $\bT_\fg^\bk(\IC(\cO, \cE))$ is a composition factor of $\bT_\fg^\bk(\bk \lotimes_\O \IC(\cO, \cE^\O))$. In particular, we deduce that this perverse sheaf is supported on $\overline{Y_{(L,\cO_L)}}$. Using again~\cite[Lemma~2.1]{genspring2} and Corollary~\ref{cor:involution=id} (now with coefficients $\bk$) it follows that $Y_{(M,\cO_M)} \subset \overline{Y_{(L,\cO_L)}}$, proving that $(L,\cO_L,\cE_L) \preceq_G (M, \cO_M, \cE_M)$. (The condition on central characters is clear by~\cite[Lemma~5.1]{genspring2}.)
\end{proof}

\begin{lem}
\label{lem:ind-series-po}
Let $L \subset M \subset G$ be Levi subgroups, and let $P \subset Q \subset G$ be corresponding parabolic subgroups.  If $(\cO,\cE) \in \fN^{(L,\cO_L,\cE_L)}_{M,\bk}$, then every composition factor of $\Ind_{M\subset Q}^G (\IC(\cO,\cE))$ lies in some series $\fN^{(N,\cO_{N},\cE_{N})}_{G,\bk}$ with $(N,\cO_{N},\cE_{N}) \succeq_G (L,\cO_L,\cE_L)$.
\end{lem}

\begin{proof}
Let $R = M \cap P$; then by assumption $\IC(\cO,\cE)$ is a quotient of $\Ind_{L \subset R}^M (\IC(\cO_L,\cE_L))$.  Since $\Ind_{M \subset Q}^G$ is exact (\cite[\S 2.1]{genspring1}) and kills no nonzero perverse sheaf (\cite[Corollary~2.15]{genspring1}), and since $\Ind_{M \subset Q}^G \Ind_{L \subset R}^M \cong \Ind_{L \subset P}^G$ (\cite[Lemma~2.6]{genspring1}), it suffices to prove that all composition factors of $\Ind_{L \subset P}^G (\IC(\cO_L,\cE_L))$ lie in series obeying the desired inequality. This fact follows from~\cite[Eq.~(2.2) and Lemmas~2.1 and~5.1]{genspring2}.
\end{proof}

%--------------------------------------------------------------------------
\subsection{$0$-cuspidal data dominated by a cuspidal datum}
%--------------------------------------------------------------------------

\begin{lem}\label{lem:0-series-total}
Assume that $G$ is semisimple of type $A$, or that $G$ is quasi-simple, simply connected, and not of type $A$.  Fix a central character $\chi$.  On the set $\fM_{G,\bk,\chi}^0$ of $0$-cuspidal data with central character $\chi$, the partial order $\preceq_G$ is a total order. Moreover, this total order can be described simply as follows: for $(L,\cO_L,\cE_L),(M,\cO_M,\cE_M)\in\fM_{G,\bk,\chi}^0$, we have $(L,\cO_L,\cE_L)\preceq_G(M,\cO_M,\cE_M)$ if and only if there is an element $g\in G$ such that $g L g^{-1} \subset M$.
\end{lem}
\begin{proof}
If $G$ is semisimple of type $A$, then Proposition~\ref{prop:0-series-A}\eqref{it:0-cusp-char} tells us that $\fM_{G,\bk,\chi}^0$ is a singleton, so the lemma is trivial in this case.

If $G$ is quasi-simple, simply connected, and of exceptional type, then, from the classification in~\cite[\S 15]{lusztig}, we see that $\fM_{G,\bk,\chi}^0$ contains one or two elements, and that when it contains two elements, one of them is of the form $(G,\cO_G,\cE_G)$; i.e., it is actually a $0$-cuspidal pair for $G$ itself.  Let $(L,\cO_L,\cE_L)$ be the other $0$-cuspidal datum with the same central character.  In every case, the classification shows that $\cO_L$ is the regular nilpotent orbit in $L$, so the induced orbit $\mathrm{Ind}_L^G \cO_L$ is the regular orbit $\cO_\reg$ for $G$. 
%According to~\cite[Proposition~6.5]{lusztig-cusp2}, we have $Y_{G,\cO_G} \subset \overline{Y_{L,\cO_L}}$, so 
By Proposition~\ref{prop:lusztig-po} this implies that
$(L,\cO_L,\cE_L) \preceq_G (G,\cO_G,\cE_G)$, as desired.  %(Note that $\preceq$ has a different meaning in~\cite{lusztig-cusp2}: it is a refinement of the opposite of the partial order denoted $\preceq_G$ here.) Moved earlier -- AH.

Suppose now that $G = \Sp(2n)$.  The pairs $(L,\cO_L)$ admitting cuspidal local systems are of the form $(L_k,\cO_k)$ where
\[
L_k =\GL(1)^{n - k(k+1)/2} \times \Sp(k(k+1)),
\qquad
\cO_{k} = \cO_0 \times \cO_{(2k,2(k-1),\cdots,4,2)},
\]
and $k$ is a nonnegative integer such that $n\geq k(k+1)/2$.
So to prove the claim in this case, it suffices to show that if $k<k'$ then $\cO_{k'}$ is contained in the closure of $\mathrm{Ind}_{L_k}^{L_{k'}}(\cO_{k})$. We can assume that $n=k'(k'+1)/2$, so that $L_{k'}=G$. According to~\cite[Theorem~7.3.3]{cm}, the induced orbit $\mathrm{Ind}_{L_k}^G \cO_{k}$ corresponds to the partition
\[
(2n - k(k+1) + 2k, 2(k-1), 2(k-2), \cdots, 4, 2),
\]
which does indeed dominate the partition $(2k',2(k'-1),2(k'-2),\cdots,4,2)$.
% It is easy to check that the dominance order on partitions of this form is a total order. This was not exactly what we wanted to prove. -- AH.

Finally, suppose $G = \Spin(n)$.  For any Levi subgroup $L \subset G$, let $\overline{L}$ be its image in $\SO(n)$.  Note that $L$ is determined by $\overline{L}$.  Consider first the case where $\chi$ is trivial.  The pairs $(L,\cO_L)$ admitting cuspidal local systems with trivial central character are of the form $(L_k,\cO_k)$ where
\[
\overline{L}_k = \GL(1)^{(n - k^2)/2} \times \SO(k^2),
\qquad
\cO_k = \cO_0 \times \cO_{(2k-1, 2k-3, \cdots, 3, 1)},
\]
and $k$ is a nonnegative integer such that $n\geq k^2$. So to prove the claim in this case, it suffices to show that if $k<k'$ then $\cO_{k'}$ is contained in the closure of $\mathrm{Ind}_{L_k}^{L_{k'}}(\cO_{k})$. We can assume that $n=(k')^2$, so that $L_{k'}=G$. 
According to~\cite[Theorem~7.3.3]{cm}, the induced orbit $\mathrm{Ind}_{L_k}^G \cO_k$ corresponds to the partition
\[
(n - k^2 + 2k-1, 2k-3, 2k-5, \cdots, 3, 1),
\]
which does indeed dominate the partition $(2k'-1,2k'-3,2k'-5,\cdots,3,1)$.
%Again, partitions of this form are totally ordered.  
For the case where $\chi$ is nontrivial, we must look at the pairs $(L_j,\cO_j)$ where
\[
\overline{L}_j = \GL(1)^{(2n - j(j+1))/4} \times \SO(j(j+1)/2),
\qquad
\cO_j = \cO_0 \times \cO_{(2j-1, 2j-5, \cdots)},
\]
and $j$ is a nonnegative integer such that $n\geq j(j+1)/2$. The argument in this case is entirely similar.
%According to~\cite[Theorem~7.3.3]{cm}, the induced orbit $\mathrm{Ind}_L^G \cO_L$ corresponds to the partition
%\[
%(n - k(k+1)/2 + 2k - 1, 2k -5, 2k-9, \cdots).
%\]
%These partitions are again totally ordered.
\end{proof}

Now we revert to considering a general connected reductive group $G$.

\begin{lem}\label{lem:0-cuspidal-order}
Let $(L,\cO_L,\cE_L)$ and $(M,\cO_M,\cE_M)$ be two $0$-cuspidal data.  Then $(L,\cO_L,\cE_L) \preceq_G (M,\cO_{M}, \cE_{M})$ if and only if the following conditions both hold:
\begin{enumerate}
\item $\cE_L$ and $\cE_M$ have the same central character.
\item There exists an element $g \in G$ such that $g L g^{-1} \subset M$.
\end{enumerate}
\end{lem}
\begin{proof}
In view of Proposition~\ref{prop:lusztig-po}, it suffices to show that for $0$-cuspidal data, the two conditions above imply that $\cO_M \subset \overline{\mathrm{Ind}_{gLg^{-1}}^M(g \cdot \cO_L)}$.  As usual, we may reduce to the case where $G$ is either semisimple of type $A$, or quasi-simple, simply connected, and not of type $A$.  Then the claim follows from Lemma~\ref{lem:0-series-total}.%We may further assume without loss of generality that $L \subset M$.
%From here on, the proof is very similar to that of Lemma~\ref{lem:0-series-total}, and comes down to calculations of induced nilpotent orbits in each type.  If $G$ is semisimple of type $A$, Proposition~\ref{prop:0-series-A} implies that there is nothing to prove.  If $G$ is of exceptional type, the only nontrivial cases have $M = G$, and we saw in the preceding proof that $\mathrm{Ind}_L^G \cO_L$ is the regular orbit for $G$.  The classical types were already handled in the proof of Lemma~\ref{lem:0-series-total}.  We omit further details.
\end{proof}

\begin{cor}
\label{cor:order-0-Levi}
Let $M \subset G$ be a Levi subgroup, and let $(L,\cO_L,\cE_L)$ and $(L',\cO_{L'}, \cE_{L'})$ be two $0$-cuspidal data for $G$.  If $M$ contains both $L$ and $L'$, then
\begin{equation}
\label{eqn:order-0-Levi}
(L,\cO_L,\cE_L) \preceq_M (L',\cO_{L'}, \cE_{L'})
\qquad\text{if and only if}\qquad
(L,\cO_L,\cE_L) \preceq_G (L',\cO_{L'}, \cE_{L'}).
\end{equation}
\end{cor}
\begin{proof}
The `only if' direction was discussed in~\eqref{eqn:order-Levi}.  By Lemma~\ref{lem:0-cuspidal-order}, the `if' direction reduces to the claim that if $L$ is $G$-conjugate to a subgroup of $L'$, then it is $M$-conjugate to a subgroup of $L'$.  Recall that a Levi subgroup that admits a $0$-cuspidal pair is \emph{self-opposed} in the sense of~\cite[\S1.E]{bonnafe1}, by~\cite[Theorem~9.2(a)]{lusztig}.  The claim about $M$-conjugacy then follows from~\cite[Proposition~1.12(d)]{bonnafe1}.
\end{proof}

\begin{cor}
\label{cor:0-series-max}
For each cuspidal datum $(M,\cO_M,\cE_M)$, 
there is (up to conjugacy) a unique maximal $0$-cuspidal datum $(L,\cO_L,\cE_L)$ such that $(L,\cO_L,\cE_L) \preceq_G (M,\cO_M,\cE_M)$. 
\end{cor}

\begin{proof}
Lemma~\ref{lem:0-series-ineq} implies that there exists at least one $0$-cuspidal datum $(L,\cO_L,\cE_L)$ such that $(L,\cO_L,\cE_L) \preceq_G (M,\cO_M,\cE_M)$.   If $G$ is semisimple of type $A$ or quasi-simple, simply connected, and not of type $A$, then the existence of a unique maximal such $0$-cuspidal datum follows from Lemma~\ref{lem:0-series-total}.  We then deduce the result for general $G$ using Lemma~\ref{lem:cogood-reduction}.
\end{proof}

%%%%%%%%%%%%%%%%%%%%%%%%%%%%%%%%%%%%%%%%%%%%%%%%%%%%%%%%%%%%%%%%%%%%%%%%%%%
\section{Comparing induction series with induction $0$-series}
\label{sec:comparison}
%%%%%%%%%%%%%%%%%%%%%%%%%%%%%%%%%%%%%%%%%%%%%%%%%%%%%%%%%%%%%%%%%%%%%%%%%%%

In this section, we prove the first main result of the paper: that each induction $0$-series is a union of induction series.  

%--------------------------------------------------------------------------
\subsection{Disjointness of $0$-series of cuspidal data}
%--------------------------------------------------------------------------

A step towards this result is the following lemma, which says that $0$-series of cuspidal data (defined in~\eqref{eqn:data-0-series}) form a partition of the set of all cuspidal data.

\begin{lem}
\label{lem:cuspidal-data-0-series}
We have
\begin{equation}
\label{eqn:representatives-cusp}
\fM_{G,\bk} = \bigsqcup_{(L,\cO_L, \cE_L) \in \fM_{G,\bk}^0} \fM_{G,\bk}^{\zs(L,\cO_L,\cE_L)}.
\end{equation}
\end{lem}

\begin{proof}
It is immediate from the definitions that every cuspidal datum $(M,\cO_M,\cE_M)$ lies in some series $\fM_{G,\bk}^{\zs(L,\cO_L,\cE_L)}$; our task is to show that the various $\fM_{G,\bk}^{\zs(L,\cO_L,\cE_L)}$ are disjoint.  Let $(L,\cO_L,\cE_L)$ and $(L',\cO_{L'},\cE_{L'})$ be two $0$-cuspidal data.  Let $(M,\cO_M,\cE_M)$ and $(M',\cO_{M'}, \cE_{M'})$ be two cuspidal data such that
\[
L \subset M, \quad
(\cO_M,\cE_M) \in \fN_{M,\bk}^{\zs(L,\cO_L,\cE_L)}
\qquad\text{and}\qquad
L' \subset M', \quad
(\cO_{M'},\cE_{M'}) \in \fN_{M',\bk}^{\zs(L,\cO_{L'},\cE_{L'})}.
\]
In other words, we have
\[
(M,\cO_M,\cE_M) \in \fM_{G,\bk}^{\zs(L,\cO_L,\cE_L)}
\qquad\text{and}\qquad
(M',\cO_{M'},\cE_{M'}) \in \fM_{G,\bk}^{\zs(L',\cO_{L'},\cE_{L'})}.
\]
Assume that $(M,\cO_M,\cE_M)$ and $(M',\cO_{M'}, \cE_{M'})$ are $G$-conjugate.    We must show that $(L,\cO_L,\cE_L)$ and $(L',\cO_{L'},\cE_{L'})$ are $G$-conjugate as well.  

Let $g \in G$ be such that $g \cdot (M,\cO_M,\cE_M) = (M',\cO_{M'},\cE_{M'})$. Then $g \cdot (L,\cO_L,\cE_L)$ and $(L',\cO_{L'},\cE_{L'})$ are $0$-cuspidal data for $M'$, and the pair $(\cO_{M'},\cE_{M'})$ belongs to both $\fN_{M',\bk}^{\zs(g \cdot (L,\cO_L,\cE_L))}$ and $\fN_{M',\bk}^{\zs(L',\cO_{L'},\cE_{L'})}$. It follows that $g \cdot (L,\cO_L,\cE_L)$ and $(L',\cO_{L'},\cE_{L'})$ are $M'$-conjugate, and hence that $(L,\cO_L,\cE_L)$ and $(L',\cO_{L'},\cE_{L'})$ are $G$-conjugate, as desired.
\end{proof}

%--------------------------------------------------------------------------
\subsection{Induction $0$-series are unions of induction series}
%--------------------------------------------------------------------------

\begin{thm}
\label{thm:0-series}
Each induction $0$-series for $G$ is a union of induction series. Specifically, for $(L,\cO_L,\cE_L) \in \fM^0_{G,\bk}$ we have
\begin{equation*}%\label{eqn:0-series}
\fN^{\zs(L,\cO_L,\cE_L)}_{G,\bk} = \bigsqcup_{(M,\cO_M,\cE_M) \in \fM_{G,\bk}^{\zs(L,\cO_L, \cE_L)}}
\fN^{(M,\cO_M,\cE_M)}_{G,\bk}.
\end{equation*}
Moreover, we have
\[
\fM_{G,\bk}^{\zs(L,\cO_L, \cE_L)} = \left\{ (M,\cO_M,\cE_M) \in \fM_{G,\bk} \,\Big|\, 
\begin{array}{c}
\textup{$(L,\cO_L,\cE_L)$ is the unique maximal $0$-cuspidal} \\
\textup{datum such that $(L,\cO_L,\cE_L) \preceq_G (M,\cO_M,\cE_M)$}
\end{array} \right\},
\]
where the uniqueness was shown in Corollary~\ref{cor:0-series-max}.
\end{thm}

\begin{proof}
In view of Lemma~\ref{lem:cogood-reduction}, it suffices to consider the cases where $G$ is either semisimple of type $A$, or else quasi-simple, simply connected, and not of type $A$.  

We introduce some notation for the sets appearing in the statement:
\begin{align*}
B^{(L,\cO_L,\cE_L)} &:= \bigsqcup_{(M,\cO_M,\cE_M) \in \fM_{G,\bk}^{\zs(L,\cO_L,\cE_L)}} \fN_{G,\bk}^{(M,\cO_M,\cE_M)}, \\
C^{(L,\cO_L,\cE_L)} &:= \left\{ (M,\cO_M,\cE_M) \in \fM_{G,\bk} \,\Big|\, 
\begin{array}{c}
\text{$(L,\cO_L,\cE_L)$ is the unique maximal $0$-cuspidal} \\
\text{datum such that $(L,\cO_L,\cE_L) \preceq_G (M,\cO_M,\cE_M)$}
\end{array} \right\}.
\end{align*}
Thus, the theorem asserts that
\begin{equation}\label{eqn:0-series-restate}
\fN_{G,\bk}^{\zs(L,\cO_L,\cE_L)} = B^{(L,\cO_L,\cE_L)}
\qquad\text{and}\qquad
\fM_{G,\bk}^{\zs(L,\cO_L, \cE_L)} = C^{(L,\cO_L,\cE_L)}.
\end{equation}
By the generalized Springer correspondence~\eqref{eqn:genspring} and Lemma~\ref{lem:cuspidal-data-0-series} we have
\begin{equation}\label{eqn:0-series-equality}
\begin{gathered}
\bigsqcup_{(L,\cO_L,\cE_L) \in \fM^0_{G,\bk}}
\fN^{\zs(L,\cO_L,\cE_L)}_{G,\bk} = 
\fN_{G,\bk} = 
\bigsqcup_{(L,\cO_L,\cE_L) \in \fM^0_{G,\bk}} 
B^{(L,\cO_L,\cE_L)},
\\
\bigsqcup_{(L,\cO_L,\cE_L) \in \fM^0_{G,\bk}}
\fM^{\zs(L,\cO_L,\cE_L)}_{G,\bk} = 
\fM_{G,\bk} = 
\bigsqcup_{(L,\cO_L,\cE_L) \in \fM^0_{G,\bk}}
C^{(L,\cO_L,\cE_L)},
\end{gathered}
\end{equation}
so it suffices to show the inclusions
\begin{equation}\label{eqn:0-series-compare-single}
\fN_{G,\bk}^{\zs(L,\cO_L,\cE_L)} \subset B^{(L,\cO_L,\cE_L)}
\qquad\text{and}\qquad
\fM_{G,\bk}^{\zs(L,\cO_L, \cE_L)} \subset C^{(L,\cO_L,\cE_L)}.
\end{equation}

%The strategy for proving these is to take certain increasing unions of the various sets in such a way that a comparison is possible, and then to carry out case-by-case counting arguments.  No longer so accurate -- AH.
Let us also define
\begin{equation}\label{eqn:0-series-union}
\begin{aligned}
\fN_{G,\bk,+}^{\zs(L,\cO_L,\cE_L)} &:=
\bigsqcup_{\substack{(L',\cO_{L'},\cE_{L'}) \in \fM_{G,\bk}^0 \\
(L,\cO_L,\cE_L) \preceq_G (L',\cO_{L'},\cE_{L'})}}
\fN_{G,\bk}^{\zs(L',\cO_{L'},\cE_{L'})},
&
B_+^{(L,\cO_L,\cE_L)} &:=
\bigsqcup_{\substack{(L',\cO_{L'},\cE_{L'}) \in \fM_{G,\bk}^0 \\
(L,\cO_L,\cE_L) \preceq_G (L',\cO_{L'},\cE_{L'})}}
B^{(L',\cO_{L'},\cE_{L'})},
\\
\fM_{G,\bk,+}^{\zs(L,\cO_L, \cE_L)} &:= 
\bigsqcup_{\substack{(L',\cO_{L'},\cE_{L'}) \in \fM_{G,\bk}^0 \\
(L,\cO_L,\cE_L) \preceq_G (L',\cO_{L'},\cE_{L'})}}
\fM_{G,\bk}^{\zs(L',\cO_{L'},\cE_{L'})},
&
C_+^{(L,\cO_L,\cE_L)} &:= 
\bigsqcup_{\substack{(L',\cO_{L'},\cE_{L'}) \in \fM_{G,\bk}^0 \\
(L,\cO_L,\cE_L) \preceq_G (L',\cO_{L'},\cE_{L'})}}
C^{(L',\cO_{L'},\cE_{L'})}.
\end{aligned}
\end{equation}
Note that
\[
C_+^{(L,\cO_L,\cE_L)}
= \{ (M,\cO_M,\cE_M) \in \fM_{G,\bk} \mid (L,\cO_L,\cE_L) \preceq_G (M,\cO_M,\cE_M) \}.
\]
Lemma~\ref{lem:0-series-ineq} and~\eqref{eqn:order-Levi} imply that 
\begin{equation}\label{eqn:0-series-compare}
\fN_{G,\bk,+}^{\zs(L,\cO_L,\cE_L)} \subset B_+^{(L,\cO_L,\cE_L)}
\qquad\text{and}\qquad
\fM_{G,\bk,+}^{\zs(L,\cO_L, \cE_L)} \subset C_+^{(L,\cO_L,\cE_L)}.
\end{equation}

Suppose first that $G$ is semisimple of type $A$. By Proposition~\ref{prop:0-series-A}\eqref{it:0-cusp-char}, any two $0$-cuspidal data are incomparable, so the disjoint unions in~\eqref{eqn:0-series-union} each contain a single term, and~\eqref{eqn:0-series-compare} reduces to the desired~\eqref{eqn:0-series-compare-single}.
%Then,~\eqref{eqn:0-series-equality} implies that these containments are equalities, as desired.

Suppose now that $G$ is quasi-simple, simply connected, and of exceptional type.  From the classification of $0$-cuspidal data in~\cite[\S 15]{lusztig}, we observe that if we fix a central character $\chi$, then there are at most two induction $0$-series with central character $\chi$.  
If $(L,\cO_L,\cE_L) \in \fM_{G,\bk}^0$ is the sole $0$-cuspidal datum with its central character, then it satisfies~\eqref{eqn:0-series-compare-single} as in type $A$.  On the other hand, if there are two $0$-cuspidal data with a given central character, then one of them consists only of a ($0$-)cuspidal pair, so~\eqref{eqn:0-series-compare} has the form
\begin{gather*}
\fN_{G,\bk}^{\zs(L,\cO_L,\cE_L)} \sqcup
\fN_{G,\bk}^{\zs(G,\cO_G,\cE_G)}
\subset 
B^{(L,\cO_L,\cE_L)} \sqcup
B^{(G,\cO_G,\cE_G)},
\\
\fM_{G,\bk}^{\zs(L,\cO_L,\cE_L)} \sqcup
\fM_{G,\bk}^{\zs(G,\cO_G,\cE_G)}
\subset 
C^{(L,\cO_L,\cE_L)} \sqcup
C^{(G,\cO_G,\cE_G)}.
\end{gather*}
By Lemma~\ref{lem:comparison-cusp}, an induction $0$-series consisting only of a $0$-cuspidal pair is also an induction series.  As a consequence, we have
$
\fN_{G,\bk}^{\zs(G,\cO_G,\cE_G)}
=
B^{(G,\cO_G,\cE_G)}
$
and
$
\fM_{G,\bk}^{\zs(G,\cO_G,\cE_G)}
=
C^{(G,\cO_G,\cE_G)}
$.
We conclude that~\eqref{eqn:0-series-compare-single} holds for all $(L,\cO_L,\cE_L) \in \fM_{G,\bk}^0$.  %As in type $A$, we deduce from~\eqref{eqn:0-series-equality} that these containments are equalities.

%Next, suppose that $G = \Sp(2n)$.  In the proof of Lemma~\ref{lem:0-series-total}, we recalled the explicit list of Levi subgroups admitting a $0$-cuspidal pair.  For each such Levi subgroup, one can count the elements on both sides of~\eqref{eqn:0-series-compare}, as in the proof of~\cite[Theorem~7.2]{genspring2}.  One finds that these sets have the same cardinality, so they are equal.  Since that equality holds for every $(L,\cO_L,\cE_L)$, downward induction with respect to $\preceq_G$ shows that~\eqref{eqn:0-series-restate} holds.

Suppose next that $G = \Sp(2n)$.  In the proof of Lemma~\ref{lem:0-series-total}, we recalled the explicit list of Levi subgroups admitting a $0$-cuspidal pair: each such Levi subgroup is of the form
\[
L_k =\GL(1)^{n - k(k+1)/2} \times \Sp(k(k+1))  \qquad\text{for k such that $0 \le k(k+1) \le 2n$},
\]
and for each $L_k$, there is a unique $0$-cuspidal datum $(L_k,\cO_k,\cE_k)$.  We also saw that $(L_j,\cO_j,\cE_j) \preceq_G (L_k,\cO_k,\cE_k)$ if and only if $j \le k$.  Recall that $N_G(L_k)/L_k$ is isomorphic to the wreath product $(\Z/2\Z) \wr \fS_{n-k(k+1)/2}$, i.e., a Coxeter group of type $C_{n-k(k+1)/2}$.  The irreducible complex representations of $(\Z/2\Z) \wr \fS_{n-k(k+1)/2}$ are parametrized by the set $\Bipart(n-k(k+1)/2)$ of bipartitions of $n-k(k+1)/2$.  By the characteristic-$0$ generalized Springer correspondence, we have
\[
|\fN_{G,\bk}^{\zs(L_k,\cO_k,\cE_k)}| = |\Irr(\C[(\Z/2\Z) \wr \fS_{n-k(k+1)/2}])| 
= |\Bipart(n-k(k+1)/2)|.
\]

We now describe $\fM_{G,\bk}^{\zs(L_k,\cO_k,\cE_k)}$.  The possibilities for cuspidal data are described in the proofs of~\cite[Theorems~7.1 and~7.2]{genspring2}: each cuspidal datum involves a  Levi subgroup of the form
\[
M_\nu = \GL(\nu_1) \times \cdots \times \GL(\nu_s) \times \Sp(k(k+1))
\]
where $\nu = (\nu_1, \cdots, \nu_s)$ is a partition of $n - k(k+1)/2$ in which each $\nu_i$ is a power of $\ell$.  (In~\cite{genspring2}, the set of such partitions was denoted by $\Part(n-k(k+1)/2,\ell)$.)  Such a Levi subgroup supports a unique cuspidal datum: $(M_{\nu}, \cO_\nu \times \cO_k, \ubk \boxtimes \cE_k)$, where $\cO_\nu$ denotes the regular orbit for $\GL(\nu_1) \times \cdots \times \GL(\nu_s)$.  Since $(\cO_\nu,\ubk)$ is part of the principal $0$-series for $\GL(\nu_1) \times \cdots \times \GL(\nu_s)$, we see that in $\fN_{M_\nu,\bk}$, the pair $(\cO_\nu \times \cO_k,\ubk \boxtimes \cE_k)$ belongs to the $0$-series of $(L_k,\cO_k,\cE_k)$. So
\begin{equation} \label{eqn:0-series-description}
\fM_{G,\bk}^{\zs(L_k,\cO_k,\cE_k)}=\{(M_{\nu}, \cO_\nu \times \cO_k, \ubk \boxtimes \cE_k)\,|\,\nu\in\Part(n-k(k+1)/2,\ell)\}.
\end{equation}
For $(M_{\nu}, \cO_\nu \times \cO_k, \ubk \boxtimes \cE_k)\in\fM_{G,\bk}^{\zs(L_k,\cO_k,\cE_k)}$, we have $(L_k,\cO_k,\cE_k) \preceq_G (M_\nu,\cO_\nu \times \cO_k, \ubk \boxtimes \cE_k)$ by Lemma~\ref{lem:0-series-ineq}.  On the other hand, $(L_{k+1},\cO_{k+1},\cE_{k+1}) \not\preceq_G (M_\nu,\cO_\nu \times \cO_k, \ubk \boxtimes \cE_k)$ because $L_{k+1}$ is not conjugate to a Levi subgroup of $M_\nu$.  This shows the second inclusion in~\eqref{eqn:0-series-compare-single}.

%This reasoning shows that $\fM_{G,\bk}^{\zs(L_k,\cO_k,\cE_k)}$ is parametrized by $\Part(n-k(k+1)/2,\ell)$.    
Given $\nu \in \Part(n-k(k+1)/2,\ell)$, we have $N_G(M_\nu)/M_\nu \cong (\Z/2\Z) \wr \fS_{\sm(\nu)}$, where $\fS_{\sm(\nu)}$ is a certain product of symmetric groups defined in~\cite[\S5.4]{genspring2}.  The irreducible $\bk$-representations of this group are labelled by a certain combinatorial set of tuples of $\ell$-regular bipartitions, denoted $\uBipart_\ell(\sm(\nu))$.  The generalized Springer correspondence for $G$ gives us that
\[
|\fN_{G,\bk}^{(M_\nu, \cO_\nu \times \cO_k, \ubk \boxtimes \cE_k)}|
= |\Irr(\bk[N_G(M_\nu)/M_\nu])| 
= |\uBipart_\ell(\sm(\nu))|.
\]
It follows that
\begin{multline*}
|B^{(L_k,\cO_k,\cE_k)}| 
= \sum_{\nu \in \Part(n-k(k+1)/2,\ell)} |\fN_{G,\bk}^{(M_\nu, \cO_\nu \times \cO_k, \ubk \boxtimes \cE_k)}|
\\
= \sum_{\nu \in \Part(n-k(k+1)/2,\ell)} |\uBipart_\ell(\sm(\nu))| = |\Bipart(n-k(k+1)/2)|,
\end{multline*}
where the last equality comes from~\cite[Eq.~(7.4)]{genspring2}.  In particular, we have shown that $|\fN_{G,\bk}^{\zs(L_k,\cO_k,\cE_k)}| = |B^{(L_k,\cO_k,\cE_k)}|$.  An immediate consequence is that
$
|\fN_{G,\bk,+}^{\zs(L_k,\cO_k,\cE_k)}| = |B_+^{(L_k,\cO_k,\cE_k)}|
$,
so the first inclusion in~\eqref{eqn:0-series-compare} must be an equality, and hence (by downward induction on $k$) we have the desired first equality in~\eqref{eqn:0-series-restate}.
%We must have
%\[
%\fN_{G,\bk,+}^{\zs(L_k,\cO_k,\cE_k)} = B_+^{(L_k,\cO_k,\cE_k)}.
%\]
%This holds for all $k$, so by downward induction on $k$, we deduce that
%\[
%\fN_{G,\bk}^{\zs(L_k,\cO_k,\cE_k)} = B^{(L_k,\cO_k,\cE_k)}.
%\]

%We now turn to the second assertion in~\eqref{eqn:0-series-restate}. Now consider a cuspidal datum $(M_\nu,\cO_\nu \times \cO_k, \ubk \boxtimes \cE_k) \in \fM_{G,\bk}^{\zs(L_k,\cO_k,\cE_k)}$.  We know by Lemma~\ref{lem:0-series-ineq} that $(L_k,\cO_k,\cE_k) \preceq_G (M_\nu,\cO_\nu \times \cO_k, \ubk \boxtimes \cE_k)$.  On the other hand, $(L_{k+1},\cO_{k+1},\cE_{k+1}) \not\preceq_G (M_\nu,\cO_\nu \times \cO_k, \ubk \boxtimes \cE_k)$ because $L_{k+1}$ is not conjugate to a subgroup of $M_\nu$.  These observations show that $\fM_{G,\bk}^{\zs(L_k,\cO_k,\cE_k)} \subset C^{(L_k,\cO_k,\cE_k)}$.  As in the previous cases,~\eqref{eqn:0-series-equality} then implies that
%\[
%\fM_{G,\bk}^{\zs(L_k,\cO_k,\cE_k)} = C^{(L_k,\cO_k,\cE_k)}.
%\]
%We have now proved both parts of~\eqref{eqn:0-series-restate} for $\Sp(2n)$.

Finally, suppose that $G = \Spin(n)$.  The proof is similar to the case of $\Sp(2n)$, using the  descriptions of Levi subgroups admitting cuspidal pairs from~\cite[\S8]{genspring2} and explicit formulas for the number of elements in each series.  We omit further details.
\end{proof}

%--------------------------------------------------------------------------
\subsection{Induction and $0$-cuspidal data}
%--------------------------------------------------------------------------

Using Theorem~\ref{thm:0-series} one can prove the following counterpart of Lemma~\ref{lem:ind-series-po}, to be used later.

\begin{cor}
\label{cor:ind-0-series-po}
Let $L \subset M \subset G$ be Levi subgroups, and let $P \subset Q \subset G$ be corresponding parabolic subgroups.  If $(\cO,\cE) \in \fN^{\zs(L,\cO_L,\cE_L)}_{M,\bk}$, then every composition factor of $\Ind_{M\subset Q}^G (\IC(\cO,\cE))$ lies in some $0$-series $\fN^{\zs(L',\cO_{L'},\cE_{L'})}_{G,\bk}$ with $(L',\cO_{L'},\cE_{L'}) \succeq_G (L,\cO_L,\cE_L)$.
\end{cor}
\begin{proof}
Let $(N,\cO_N,\cE_N)$ be a cuspidal datum for $M$ such that $(\cO,\cE) \in \fN_{M,\bk}^{(N,\cO_N,\cE_N)}$. By Lemma~\ref{lem:0-series-ineq} we have $(L,\cO_L,\cE_L) \preceq_M (N,\cO_N, \cE_N)$. Using~\eqref{eqn:order-Levi} we deduce that $(L,\cO_L,\cE_L) \preceq_G (N,\cO_N, \cE_N)$. Consider now a composition factor $\cF$ of $\Ind_{M\subset Q}^G (\IC(\cO,\cE))$, and let $(K,\cO_K,\cE_K)$ be a cuspidal datum such that the pair associated to $\cF$ belongs to $\fN_{G,\bk}^{(K,\cO_K,\cE_K)}$. By Lemma~\ref{lem:ind-series-po} we have $(N,\cO_N, \cE_N) \preceq_G (K,\cO_K,\cE_K)$, hence \textit{a fortiori} $(L,\cO_L,\cE_L) \preceq_G (K,\cO_K,\cE_K)$. Now if the pair associated to $\cF$ belongs to $\fN_{G,\bk}^{\zs(L',\cO_{L'}, \cE_{L'})}$, then by Theorem~\ref{thm:0-series}, $(L',\cO_{L'}, \cE_{L'})$ is characterized by the property that it is the unique (up to $G$-conjugation) maximal $0$-cuspidal datum which is smaller than $(K,\cO_K,\cE_K)$ for $\preceq_G$. Hence we must have $(L,\cO_L,\cE_L) \preceq_G (L',\cO_{L'}, \cE_{L'})$.
\end{proof}

%%%%%%%%%%%%%%%%%%%%%%%%%%%%%%%%%%%%%%%%%%%%%%%%%%%%%%%%%%%%%%%%%%%%%%%%%%%
\section{Cleanness conjecture}
\label{sec:clean}
%%%%%%%%%%%%%%%%%%%%%%%%%%%%%%%%%%%%%%%%%%%%%%%%%%%%%%%%%%%%%%%%%%%%%%%%%%%

%--------------------------------------------------------------------
\subsection{Cleanness}
%--------------------------------------------------------------------

Following~\cite[Definition~7.7]{charsh2}, we say that a 
simple object $\IC(\cO,\cE)$ in $\Perv_G(\cN_G,\bk)$ is \emph{clean} if its restriction to $\overline{\cO} \smallsetminus \cO$ vanishes.  In characteristic $0$, all simple perverse sheaves associated to cuspidal pairs are clean~\cite[\S 23]{charsh5}. This fact plays a key role in the theory of character sheaves. It is known that this property does not hold in general when $\ell>0$ (for example, see~\cite[Remark 2.5]{genspring1}). However, C.~Mautner conjectured that it holds for \emph{$0$-cuspidal pairs} in rather good characteristic. 

\begin{conj}[Mautner's cleanness conjecture]
\label{conj:clean}
If $\ell$ is rather good for $G$, then for every pair $(\cO,\cE) \in \fN^\zcusp_{G,\bk}$, the simple perverse sheaf $\IC(\cO,\cE)$ is clean.
\end{conj}

In other words, this conjecture asserts that for $(\cO,\cE) \in \fN^\zcusp_{G,\bk}$ we have $\IC(\cO,\cE) \cong  j_{\cO!}\cE[\dim \cO]$, where $j_\cO: \cO \hookrightarrow \cN_G$ is the inclusion map. Since the class of $0$-cuspidal perverse sheaves is stable under Verdier duality~\cite[Remark~2.3]{genspring1}, the conjecture is equivalent to the apparently stronger assertion that for any $(\cO,\cE) \in \fN^\zcusp_{G,\bk}$ we have
\begin{equation}
\label{eqn:clean-restate}
\IC(\cO,\cE) \cong j_{\cO*}\cE[\dim \cO] \cong j_{\cO!}\cE[\dim \cO].
\end{equation}

It will be useful to restate the cleanness condition in terms of $\Hom$-vanishing. Let $i: \overline{\cO} \smallsetminus \cO \hookrightarrow \cN_G$ denote the inclusion map. By adjunction, $\IC(\cO,\cE)$ is clean if and only if
\begin{equation}
\label{eqn:clean-hom-general}
\Hom_{\Db_G(\cN_G,\bk)}(\IC(\cO,\cE),i_*\cF)=0\quad\text{for all }\cF\in \Db_G(\overline{\cO} \smallsetminus \cO,\bk).
\end{equation}
It clearly suffices to verify the vanishing in~\eqref{eqn:clean-hom-general} for a set of objects $\cF$ that generate $\Db_G(\overline{\cO} \smallsetminus \cO,\bk)$ as a triangulated category. For instance, $\IC(\cO,\cE)$ is clean if and only if
\begin{equation}
\label{eqn:clean-hom-IC}
\Hom_{\Db_G(\cN_G,\bk)}^k(\IC(\cO,\cE),\IC(\cO',\cE'))=0\quad\text{for all $k \in \Z$ and all }(\cO',\cE')\in\fN_{G,\bk}\text{ with }\cO'\subset\overline{\cO}, \cO'\neq\cO.
\end{equation}
Here, as usual, $\Hom_{\Db_G(\cN_G,\bk)}^k(A,B)$ denotes $\Hom_{\Db_G(\cN_G,\bk)}(A,B[k])$.

%--------------------------------------------------------------------
\subsection{Verification of the cleanness conjecture in some cases}
%--------------------------------------------------------------------

In this section and the following one, we will show that the cleanness conjecture holds in the following cases.

\begin{thm}\label{thm:clean}
Conjecture~{\rm \ref{conj:clean}} holds in the following cases:
\begin{enumerate}
\item 
\label{it:conj-clean-easy}
$\ell \nmid |W|$;
\item 
\label{it:conj-clean-A}
$G/Z(G)^\circ$ is semisimple of type $A$;
\item 
\label{it:conj-clean-C3}
$G$ is quasi-simple and simply connected of type $B_4$, $C_3$ or $D_5$;
\item
\label{it:conj-clean-exc}
$G$ is quasi-simple and simply connected of exceptional type.
\end{enumerate}
\end{thm}

Of course, by the principle of Lemma~\ref{lem:cogood-reduction}, the conjecture also holds for any reductive group whose root system only contains factors of the types indicated above. The conjecture also holds vacuously if there are no $0$-cuspidal pairs.

\begin{rmk} \label{rmk:bcd-cuspidal}
Recall from~\cite[Introduction]{lusztig} that a quasi-simple and simply connected group of type $B$/$C$/$D$ has a $0$-cuspidal pair only in the following circumstances:
\begin{itemize}
\item $B_n$: when $2n+1$ is either a triangular number or a square;
\item $C_n$: when $n$ is a triangular number;
\item $D_n$: when $2n$ is either a triangular number or a square.
\end{itemize}
So the cases of these types listed in Theorem~\ref{thm:clean} represent the smallest ranks for which the conjecture is not vacuously true, and the next-smallest such ranks would be $B_7$, $C_6$, $D_8$.
\end{rmk} 

\begin{proof}[Proof of Theorem~{\rm \ref{thm:clean}} in case~\eqref{it:conj-clean-A}]
For any $(\cO,\cE) \in \fN_{G,\bk}$, any local system occurring in any cohomology sheaf $\cH^i(\IC(\cO,\cE)_{|\cO'})$ must have the same central character as $\cE$.  In type $A$, if $(\cO,\cE)$ is a $0$-cuspidal pair, then in fact $(\cO,\cE)$ is the unique element of $\fN_{G,\bk}$ with its central character (a special case of Proposition~\ref{prop:0-series-A}(\ref{it:0-ser-char})), so the sheaves $\cH^i(\IC(\cO,\cE)_{|\cO'})$ must vanish for all $\cO' \ne \cO$.  In other words, $\IC(\cO,\cE)$ is clean.
\end{proof}

Another important case is provided by the following proposition.

\begin{prop}
\label{prop:clean-lusztig}
Let $(\cO,\cE)$ be a $0$-cuspidal pair with central character $\chi$. Assume that for any cuspidal datum $(L,\cO_L,\cE_L)$ where $\cE_L$ has central character $\chi$, $\ell$ does not divide $|N_G(L)/L|$. Then $\IC(\cO,\cE)$ is clean.
\end{prop}

\begin{proof}
The proof is essentially the same as one of Lusztig's arguments in the setting of character sheaves with $\ell=0$ (see the proof of~\cite[Proposition 7.9]{charsh2}), but we will express it in a form which is closer to the proof of~\cite[Proposition 4.2]{rr}. %and since our situation is simpler than the situation considered by Lusztig.

We must prove the $\Hom$-vanishing statement~\eqref{eqn:clean-hom-IC} for every pair $(\cO',\cE')\in\fN_{G,\bk}$ with $\cO'\subset\overline{\cO}$ and $\cO'\neq\cO$. In fact we will prove that for \emph{all} pairs $(\cO',\cE') \in \fN_{G,\bk}$ distinct from $(\cO,\cE)$ we have
\begin{equation*}
\Hom_{\Db_G(\cN_G,\bk)}^k(\IC(\cO,\cE),\IC(\cO',\cE'))=0\quad\text{for all $k \in \Z$.}
\end{equation*}
By Lemma~\ref{lem:central-char-1}, we need only consider pairs with the same central character $\chi$ as $(\cO,\cE)$.

By Proposition~\ref{prop:unicity-cuspidal}, a pair $(\cO',\cE')$ with central character $\chi$ and distinct from $(\cO,\cE)$ cannot be cuspidal, so it must belong to the induction series associated to a cuspidal datum $(L,\cO_L,\cE_L)$ where $L\neq G$ and $\cE_L$ has central character $\chi$. 
Since $\ell$ does not divide $|N_G(L)/L|$, the induced perverse sheaf $\Ind_{L \subset P}^G(\IC(\cO_L,\cE_L))$ is semisimple by~\cite[\mgsemisimple]{genspring3}. (Here, as usual, $P$ denotes a parabolic subgroup with Levi factor $L$.) Hence $\IC(\cO',\cE')$ is a direct summand of $\Ind_{L \subset P}^G(\IC(\cO_L,\cE_L))$, and it suffices to prove that
\begin{equation*}
\Hom_{\Db_G(\cN_G,\bk)}^k(\IC(\cO,\cE),\Ind_{L \subset P}^G(\IC(\cO_L,\cE_L)))=0\quad\text{for all $k$.}
\end{equation*}
But this is immediate from adjunction, since ${}'\Res_{L \subset P}^G(\IC(\cO,\cE))=0$ by definition of cuspidality.
\end{proof}

\begin{proof}[Proof of Theorem~{\rm \ref{thm:clean}} in case~\eqref{it:conj-clean-easy} and most of cases~\eqref{it:conj-clean-C3} and~\eqref{it:conj-clean-exc}]
We can now prove the cleanness conjecture in the remaining cases in the theorem, excluding type $E_8$ in characteristic $7$ and type $B_4$ in characteristic $3$.

In case~\eqref{it:conj-clean-easy}, the fact that $\ell \nmid |W|$ implies that for any Levi subgroup $L \subset G$, $\ell$ does not divide the cardinality of the group $N_G(L)/L$, since the latter is a subquotient of $W$ (see, for instance,~\cite[\mgeqnweylisom]{genspring3}). The conjecture then follows from Proposition~\ref{prop:clean-lusztig}.

%In case~\eqref{prop:clean-lusztig}, the conjecture follows from Proposition~\ref{prop:clean-lusztig}.

If $G$ is of type $G_2$ or $F_4$ and $\ell$ is rather good, then $\ell$ does not divide $|W|$, and so the conjecture holds.
%Hence, for any Levi subgroup $L \subset G$, $\ell$ does not divide the cardinality of the subquotient $N_G(L)/L$ (see, for instance,~\cite[\mgeqnweylisom]{genspring3}). We deduce the result again from Proposition~\ref{prop:clean-lusztig} in these cases.

If $G$ is simply connected of type $C_3$, resp.~$D_5$, $E_6$, $E_7$, and $\ell$ is rather good, there is one $0$-cuspidal pair for each faithful central character $\chi$, and there are no other cuspidal pairs. Moreover there is (up to conjugacy) only one other cuspidal datum $(L,\cO_L,\cE_L)$ with central character $\chi$, namely the minimal one described in~\cite[\mgmincuspdatum]{genspring3}, in which the Levi subgroup $L$ is of type $A_1$, resp.~$A_1+A_3$, $2A_2$, $(3A_1)''$. The associated group $N_G(L)/L$ is isomorphic to $W(C_2)$, resp.~$W(A_1)$, $W(G_2)$, $W(F_4)$. The rather good prime $\ell$ does not divide the cardinality of this group, so once again the result follows from Proposition~\ref{prop:clean-lusztig}.

Finally, if $G$ is of type $B_4$ and $\ell>3$, or if $G$ is of type $E_8$ and $\ell>7$, then $\ell$ does not divide $|W|$, and we can conclude as before.
\end{proof}

%%%%%%%%%%%%%%%%%%%%%%%%%%%%%%%%%%%%%%%%%%%%%%%%%%%%%%%%%%%%%%%%%%%%%%%%%%%
\section{Proof of the cleanness conjecture in types $E_8$ and $B_4$}
\label{sec:e8-clean}
%%%%%%%%%%%%%%%%%%%%%%%%%%%%%%%%%%%%%%%%%%%%%%%%%%%%%%%%%%%%%%%%%%%%%%%%%%%

The only two remaining cases in Theorem~\ref{thm:clean} are that in which $G$ is quasi-simple of type $E_8$ and $\ell = 7$, and that in which $G$ is quasi-simple of type $B_4$ and $\ell = 3$.  This section is devoted to the proof of these cases. The key property they share is that, if $(\cO,\cE)$ denotes the unique $0$-cuspidal pair, all (or, in the latter case, almost all) the pairs $(\cO',\cE')\in\fN_{G,\bk}$ with $\cO'\subset\overline{\cO}$ and $\cO'\neq\cO$ belong to the modular Springer correspondence.

%--------------------------------------------------------------------
\subsection{Image of the Springer correspondence}
%--------------------------------------------------------------------

We begin with some results that apply to arbitrary connected reductive groups and arbitrary rather good primes.  Let $\fg_{\mathrm{rs}} \subset \fg$ be the open subset consisting of regular semisimple elements.  This set coincides with the Lusztig stratum $Y_{(T,\{0\})}$, where $T$ is a maximal torus in $G$.  Recall that over $\fg_{\mathrm{rs}}$, the Grothendieck--Springer simultaneous resolution of $\fg$ restricts to a Galois covering map with Galois group $W$.  Thus, if $\E$ is any member of the $\ell$-modular triple $(\K, \O, \bk)$, then any $\E[W]$-module $V$ determines an $\E$-local system $\cL_V$ on $\fg_{\mathrm{rs}}$ (see, for instance,~\cite[Eq.~(2.16)]{genspring1}).  We define a functor
\[
\Psi_\E: \Rep(\E[W]) \to \Perv_G(\cN_G,\E)
\qquad\text{by}\qquad
\Psi_\E(V) = (\bT_\fg)^{-1}(\IC(\fg_{\mathrm{rs}}, \cL_V)),
\]
where $\bT_\fg$ is the Fourier--Sato transform on $\fg$.  
%According to~\cite[Proposition~5.4]{wgasp}, 
By construction,
if $\E$ is a field and $V$ is irreducible, then $\Psi_\E(V)$ is the simple perverse sheaf that corresponds to $V$ under the bijection~\eqref{eqn:genspring-series} in the special case of the `principal' cuspidal datum $(L,\cO_L,\cE_L)=(T,\{0\},\underline{\E})$, in other words under the Springer correspondence. 
%(Here $T$ is a maximal torus in $G$.) 
We denote by $\Perv_{G}^{\mathrm{Spr}}(\cN_G, \E)$ the essential image of this functor. Since $\Psi_\E$ is fully faithful (because all of the functors $V \mapsto \cL_V$, $\IC(\fg_{\mathrm{rs}}, -)$ and $(\bT_\fg)^{-1}$ are fully faithful), it induces an equivalence $\Rep(\E[W])\cong\Perv_{G}^{\mathrm{Spr}}(\cN_G, \E)$.

\begin{lem}
\label{lem:image-springer}
For $\E\in\{\K, \O, \bk\}$, define $\Perv_{G}^{\mathrm{Spr}}(\cN_G, \E)$ as above.
\begin{enumerate}
\item
\label{it:springer-K-O-k}
Let $\cF \in \Perv_G(\cN_G, \O)$ be torsion-free. Assume that $\K \otimes_\O \cF$ belongs to $\Perv_{G}^{\mathrm{Spr}}(\cN_G, \K)$ and that $\bk \lotimes_\O \cF$ belongs to $\Perv_{G}^{\mathrm{Spr}}(\cN_G, \bk)$. Then $\cF$ belongs to $\Perv_{G}^{\mathrm{Spr}}(\cN_G, \O)$.
\item
\label{it:image-springer-field}
If $\E=\K$ or $\bk$,
an object $\cF$ of $\Perv_{G}(\cN_G, \E)$ belongs to $\Perv_{G}^{\mathrm{Spr}}(\cN_G, \E)$ if and only if it has no subobject or quotient of the form $\IC(\cO,\cE)$ with $(\cO,\cE) \notin \fN_{G,\E}^{(T,\{0\},\underline{\E})}$.
\end{enumerate}
\end{lem}

\begin{proof}
\eqref{it:springer-K-O-k} We consider the torsion-free $\O$-perverse sheaf $\cG:=\bT_\fg(\cF)$. Since $\bk \lotimes_\O \cF$ belongs to $\Perv_{G}^{\mathrm{Spr}}(\cN_G, \bk)$, the perverse sheaf $\bk \lotimes_\O \cG$ is the $\IC$-extension of a local system on $\fg_{\mathrm{rs}}$. Therefore, if $i : \fg \smallsetminus \fg_{\mathrm{rs}} \hookrightarrow \fg$ is the inclusion, we have
\[
\bk \lotimes_\O i^*(\cG) \in {}^p \hspace{-1pt} D^{\leq -1}_G(\cN_G, \bk) \qquad \text{and} \qquad \bk \lotimes_\O i^!(\cG) \in {}^p \hspace{-1pt} D^{\geq 1}_G(\cN_G, \bk).
\]
It follows that we also have
\[
i^*(\cG) \in {}^p \hspace{-1pt} D^{\leq -1}_G(\cN_G, \O) \qquad \text{and} \qquad i^!(\cG) \in {}^p \hspace{-1pt} D^{\geq 1}_G(\cN_G, \O),
\]
so $\cG$ is also an $\IC$-extension of a perverse sheaf $\cG_{\mathrm{rs}}$ on $\fg_{\mathrm{rs}}$. From the fact that $\bk \lotimes_\O \cG_{\mathrm{rs}}$ is a local system, one can easily deduce that $\cG_{\mathrm{rs}}$ is an $\O$-free local system. And since the monodromy action on $\K \otimes_\O \cG_{\mathrm{rs}}$ factors through an action $W$ (because $\K \otimes_\O \cF$ belongs to $\Perv_{G}^{\mathrm{Spr}}(\cN_G, \K)$), the same must be true for $\cG_{\mathrm{rs}}$. We deduce that $\cF$ belongs to $\Perv_{G}^{\mathrm{Spr}}(\cN_G, \O)$.

\eqref{it:image-springer-field}  If $(L,\cO_L,\cE_L)$ is any cuspidal datum other than $(T,\{0\},\underline{\E})$, then the corresponding Lusztig stratum $Y_{(L,\cO_L)}$ (see Section~\ref{sec:partialorder} or~\cite[\S 2.1]{genspring2}) is contained in $\fg \smallsetminus \fg_{\mathrm{rs}}$.  By~\cite[Lemma~2.1]{genspring2}, a simple perverse sheaf $\IC(\cO,\cE) \in \Perv_G(\cN_G,\E)$ has $(\cO,\cE) \notin \fN_{G,\bk}^{(T,\{0\},\underline{\E})}$ if and only if $\bT_\fg(\IC(\cO,\cE))$ is supported on $\fg \smallsetminus \fg_{\mathrm{rs}}$.  
%(To be precise,~\cite[Lemma~2.1]{genspring2} tells us instead about the support of $\bT_\fg(\IC(\cO,\cE))$, but since each Lusztig stratum $Y_{L,\cO_L}$ is stable under the map $x \mapsto -x$, we see from~\cite[Eq.~(3.7.9)]{ks} that $\bT_\fg(\IC(\cO,\cE))$ and $\bT_\fg^{-1}(\IC(\cO,\cE))$ have the same support.)  
Thus, part~\eqref{it:image-springer-field} is equivalent to the assertion that $\cF$ belongs to $\Perv_{G}^{\mathrm{Spr}}(\cN_G, \E)$ if and only if $\bT_\fg(\cF)$ has no subobject or quotient supported on $\fg \smallsetminus \fg_{\mathrm{rs}}$.  For any $\cF \in  \Perv_{G}(\cN_G, \E)$,
%~\cite[Lemma~2.1]{genspring2} implies that 
we know by~\cite[Lemma~6.1]{mautner} that $\bT_\fg(\cF)_{|\fg_{\mathrm{rs}}}[-\dim \fg_{\mathrm{rs}}]$ is a local system (possibly $0$) with monodromy action factoring through $W$. Hence $\bT_\fg(\cF)$ has no subobject or quotient supported on $\fg \smallsetminus \fg_{\mathrm{rs}}$ if and only if $\bT_\fg(\cF)$ is isomorphic to $\IC(\fg_{\mathrm{rs}}, \cL)$ for some local system $\cL$ on $\fg_{\mathrm{rs}}$ arising from a representation of $W$. This proves the claim. 
%So the statement we wish to prove reduces to the claim that $\cF$ belongs to $\Perv_{G}^{\mathrm{Spr}}(\cN_G, \E)$ if and only if $\bT_\fg(\cF)$ is isomorphic to $\IC(\fg_{\mathrm{rs}}, \cL)$ for some local system $\cL$ arising from a representation of $W$.  This claim is true by definition.
\end{proof}

Lemma~\ref{lem:image-springer} has the following consequences, which we will use crucially in \S\S\ref{ss:cleanness-E8}--\ref{ss:cleanness-B4}.

\begin{cor}
\label{cor:image-springer}
Let $X \subset \cN_G$ be a closed union of $G$-orbits.
\begin{enumerate}
\item
\label{it:PervSpr1}
Let $\E=\K$ or $\bk$, and assume that any $(\cO,\cE) \in \fN_{G,\E}$ with $\cO \subset X$ belongs to the principal series $\fN_{G,\E}^{(T,\{0\},\underline{\E})}$. Then any $\cF$ in $\Perv_G(\cN_G,\E)$ supported on $X$ belongs to $\Perv_{G}^{\mathrm{Spr}}(\cN_G, \E)$.
\item
\label{it:PervSpr2}
Assume that for \emph{both} $\K$ and $\bk$, any pair $(\cO,\cE)$ with $\cO \subset X$ belongs to the principal series. Then any torsion-free $\cF$ in $\Perv_G(\cN_G,\O)$ supported on $X$ belongs to $\Perv_{G}^{\mathrm{Spr}}(\cN_G, \O)$.
\end{enumerate}
\end{cor}

\begin{proof}
Part~\eqref{it:PervSpr1} follows from Lemma~\ref{lem:image-springer}\eqref{it:image-springer-field}, since in this setting any simple quotient or subobject of $\cF$ must be supported on $X$, hence must correspond to a pair in the principal series. Then Part~\eqref{it:PervSpr2} is an immediate application of Lemma~\ref{lem:image-springer}\eqref{it:springer-K-O-k}.
\end{proof}

\begin{rmk}
More generally, when $\E=\K$ or $\bk$, Lemma~\ref{lem:image-springer}\eqref{it:image-springer-field} implies that the category $\Perv_G^{\mathrm{Spr}}(\cN_G, \E)$ is stable under extensions. However, in general it is not stable under taking subobjects or quotients.
\end{rmk}

Below we will also use the following result.

\begin{lem}
\label{lem:Psi-mod-red}
Let $V \in \Rep(\O[W])$ be torsion-free, and assume that $\bk \lotimes_\O \Psi_\O(V)$ belongs to $\Perv_G^{\mathrm{Spr}}(\cN_G,\bk)$. Then we have $\bk \lotimes_\O \Psi_\O(V) \cong \Psi_\bk(\bk \otimes_\O V)$.
\end{lem}

\begin{proof}
We have
\[
\bk \lotimes_\O \Psi_\O(V) = \bk \lotimes_\O (\bT_\fg)^{-1}(\IC(\fg_{\mathrm{rs}}, \cL_V)) \cong (\bT_\fg)^{-1}(\bk \lotimes_\O \IC(\fg_{\mathrm{rs}}, \cL_V))
\]
since Fourier transform commutes with modular reduction, see~\cite[Remark~2.23]{genspring1}. Now the assumption that $\bk \lotimes_\O \Psi_\O(V)$ belongs to $\Perv_G^{\mathrm{Spr}}(\cN_G,\bk)$ implies that $\bk \lotimes_\O \IC(\fg_{\mathrm{rs}}, \cL_V)$ is the $\IC$-extension of a local system on $\fg_{\mathrm{rs}}$. This local system must be the appropriate shift of the restriction of this perverse sheaf to $\fg_\mathrm{rs}$, i.e.~$\bk \otimes_\O \cL_V \cong \cL_{\bk \otimes_\O V}$. It follows that we have
\[
\bk \lotimes_\O \Psi_\O(V) \cong (\bT_\fg)^{-1}(\IC(\fg_{\mathrm{rs}}, \cL_{ \bk \otimes_\O V})),
\]
which finishes the proof.
\end{proof}

%--------------------------------------------------------------------
\subsection{Blocks}
%--------------------------------------------------------------------

Let us consider the decomposition $1 = \sum_{i=1}^s e_i$ of $1$ as a sum of orthogonal primitive idempotents in the center $Z(\O[W])$ of the group algebra $\O[W]$. This decomposition determines the decompositions of $\Irr(\K[W])$ and $\Irr(\bk[W])$ into $\ell$-blocks, see~\cite[\S 56.B]{cr}: an irreducible representation of $W$ over $\K$, resp.~over $\bk$, belongs to the block $\mathsf{B}_i$ if and only if it is a composition factor of the representation $\K[W] \cdot e_i$, resp.~$\bk[W] \cdot \overline{e}_i$, where $\overline{e}_i \in \bk[W]$ is the image of $e_i$ in $\bk[W]$ (a primitive central idempotent). Equivalently, the irreducible representation $V$ over $\K$, resp.~$\bk$, belongs to $\mathsf{B}_i$ iff $e_i$, resp.~$\overline{e}_i$, acts as the identity on $V$. 

Recall (see~\cite[Proposition~56.31]{cr}) that if $V$ is an irreducible $\K$-representation of $W$ such that $\dim V$ is divisible by the largest power of $\ell$ dividing $|W|$, then $V$ is the unique irreducible $\K$-representation belonging to its block; such blocks are said to be of \emph{defect $0$}. If $\mathsf{B}_i$ is a block of defect $0$, we write $E_i^\K$ for its unique irreducible $\K$-representation.   
%The `simplest' blocks are those of defect $0$, i.e.~the blocks $\mathsf{B}_i$ such that the largest power of $\ell$ dividing $|W|$ divides the dimension of all $\K$-representations belonging to $\mathsf{B}_i$. In this case, only one irreducible $\K$-representation $E_i^\K$ belongs to $\mathsf{B}_i$,
If $E_i^\O$ is an $\O$-form of $E_i^\K$, then $E_i^\O$ is a projective $\O[W]$-module, $E_i^{\bk} :=\bk \otimes_\O E_i^\O$ is a simple and projective $\bk[W]$-module, and $E_i^{\bk}$ is the unique irreducible $\bk$-representation belonging to $\mathsf{B}_i$. 

Recall also (see~\cite[\S 56.26]{cr}) that the decomposition matrix is block-diagonal if irreducible representations are ordered by $\ell$-blocks; in particular, if $\mathsf{B}_i$ is a block of defect $0$, with the notation above $E_i^\K$ is the only irreducible $\K$-representation whose modular reduction has $E_i^\bk$ as a composition factor.

\begin{lem}
\label{lem:blocks-defect-0}
Let $\mathsf{B}_i$ be a block of defect $0$, and let $(\cO,\cE^\K)$, resp.~$(\cO', \cE^\bk)$, be the pair corresponding to $E_i^\K$, resp.~$E_i^\bk$, under the Springer correspondence. Then $\theta_G(\cO,\cE^\K)=(\cO', \cE^\bk)$ (in particular, $\cO'=\cO)$, and $\IC(\cO', \cE^\bk)$ is a direct summand of the Springer sheaf $\Psi_\bk(\bk[W])$.
\end{lem}

\begin{proof}
The first assertion follows from~\cite[Theorem~5.3]{juteau}. In fact, as recalled above, $E_i^\K$ is the only irreducible $\K$-representation whose modular reduction has $E_i^\bk$ as a composition factor so that, using the notation of that result, we must have $\beta_S(E_i^\bk)=E_i^\K$. Similarly, the `rather good' assumption forces the choice $\beta_{\cN} ( \theta_G(\cO,\cE^\K))=(\cO, \cE^\K)$.

For the second assertion, we simply remark that
since $E_i^\bk$ is a direct summand of $\bk[W]$, the perverse sheaf $\IC(\cO', \cE^\bk)=\Psi_\bk(E_i^\bk)$ is a direct summand of $\Psi_\bk(\bk[W])$.
\end{proof}

Using the equivalence $\Rep(\E[W])\cong\Perv_{G}^{\mathrm{Spr}}(\cN_G, \E)$ induced by $\Psi_\E$, we obtain an action of the centre $Z(\E[W])$ on $\Perv_{G}^{\mathrm{Spr}}(\cN_G, \E)$. So if $\cF$ belongs to $\Perv_{G}^{\mathrm{Spr}}(\cN_G, \O)$ or to $\Perv_{G}^{\mathrm{Spr}}(\cN_G, \K)$, then it makes sense to consider the object $e_i \cdot \cF$, a direct summand of $\cF$. Similarly, if $\cF$ belongs to $\Perv_{G}^{\mathrm{Spr}}(\cN_G, \bk)$, then it makes sense to consider the direct summand $\overline{e}_i \cdot \cF$ of $\cF$. 

\begin{lem}
\label{lem:compatibility}
If $\cF \in \Perv_G^{\mathrm{Spr}}(\cN_G, \O)$ is torsion-free and if $\bk \lotimes_\O \cF$ belongs to $\Perv_G^{\mathrm{Spr}}(\cN_G, \bk)$, then we have
\[
\overline{e}_i \cdot (\bk \lotimes_\O\cF)\cong \bk\lotimes_\O(e_i\cdot\cF).
\]
\end{lem}
 
\begin{proof}
Let $V \in \Rep(\O[W])$ be such that $\cF \cong \Psi_\O(V)$. Then
by Lemma~\ref{lem:Psi-mod-red} (applied twice) we have
\[
\overline{e}_i \cdot (\bk \lotimes_\O\cF) \cong \overline{e}_i \cdot \Psi_\bk(\bk \lotimes_\O V) = \Psi_\bk(  \overline{e}_i \cdot (\bk \lotimes_\O V)) \cong \Psi_\bk( \bk \lotimes_\O (e_i \cdot V)) \cong \bk \lotimes_\O \Psi_\O(e_i \cdot V) = \bk\lotimes_\O(e_i\cdot\cF),
\]
which proves the claim.
\end{proof}

%--------------------------------------------------------------------
\subsection{Proof of Theorem~\ref{thm:clean} for $G$ of type $E_8$ and $\ell = 7$}
\label{ss:cleanness-E8}
%--------------------------------------------------------------------

In this proof, for brevity, we will denote nilpotent orbits and Levi subgroups by their Bala--Carter labels, see~\cite{cm} (except for the trivial orbit, denoted $\{0\}$), %and the maximal torus, denoted $T$), 
and we will follow the notation of~\cite{juteau} for nontrivial local systems, with a superscript indicating the coefficient ring.  Thus, the unique $0$-cuspidal pair in $\fN_{G,\bk}$ is the pair $(E_8(a_7),11111^\bk)$.  The corresponding simple perverse sheaf is denoted $\IC(E_8(a_7),11111^\bk)$.

Let $X$ be the union of all nilpotent orbits for $G$ that are contained in the closure of $E_8(a_7)$ but different from it.  It suffices to prove the following claim: \textit{The category $\Db_G(X,\bk)$ is generated as a triangulated category by direct summands of perverse sheaves induced from proper Levi subgroups.}  Indeed, the $\Hom$-vanishing statement~\eqref{eqn:clean-hom-general} holds whenever $\cF$ is a direct summand of a perverse sheaf induced from a proper Levi subgroup, by adjunction. In contrast to the proof of Proposition~\ref{prop:clean-lusztig}, the set of direct summands of induced perverse sheaves we will use to generate $\Db_G(X,\bk)$ will not consist entirely of simple perverse sheaves.
%\begin{equation}\label{eqn:e8-clean-SR}
%\Hom^\bullet(\IC(E_8(a_7),11111^\bk),\cF) = \Hom^\bullet(\cF,\IC(E_8(a_7),11111^\bk)) = 0
%\end{equation}
%for any $\cF \in \Db_G(\cN_G,\bk)$ supported on $X$ (because it holds when $\cF$ is an induced perverse sheaf, by adjunction). It is clear that~\eqref{eqn:e8-clean-SR} is equivalent to the cleanness of $\IC(E_8(a_7),11111^\bk)$.

We use the determination of the modular Springer correspondence and its decomposition into blocks given in~\cite[\S9.5.4]{juteau}. There are fifty-nine pairs $(\cO,\cE)$ with $\cO \subset X$, all of which are in the principal series $\fN_{G,\E}^{(T,\{0\}, \underline{\E})}$ for both $\E=\K$ and $\E=\bk$. So Corollary~\ref{cor:image-springer} guarantees that any perverse sheaf supported on $X$ is in $\Perv_G^{\mathrm{Spr}}(\cN_G, \E)$ for $\E \in \{\K,\O,\bk\}$, provided $\cF$ is torsion-free in case $\E=\O$.  Among the fifty-nine pairs in $\fN_{G,\bk}$ supported on $X$, forty-five correspond to $\bk[W]$-representations in blocks of defect $0$:
\[
\begin{array}{ccccccc}
(E_7(a_5),\ubk) & (A_4+A_3,\ubk) & (D_5(a_1),\ubk) & (D_4(a_1)+A_1,21^\bk)   &  (2A_2+A_1,\ubk)  \\
(E_7(a_5),21^\bk)  & (D_5,\ubk) & (D_5(a_1),11^\bk) & (D_4(a_1)+A_1,111^\bk) &  (2A_2,\ubk)  \\
(E_7(a_5),111^\bk) & (E_6(a_3),\ubk) &  (2A_3,\ubk)  & (A_3+2A_1,\ubk) & (A_3,\ubk)  \\
(E_6(a_3)+A_1,\ubk)  &  (D_4+A_2,\ubk)  &  (D_4(a_1)+A_2,\ubk)  & (2A_2+2A_1,\ubk) & (A_2+2A_1,\ubk) \\
(E_6(a_3)+A_1,11^\bk) &  (D_4+A_2,11^\bk) &  (D_4(a_1)+A_2,11^\bk) & (D_4,\ubk) & (A_2+A_1,\ubk)  \\
(D_6(a_2),\ubk) & (A_4+A_2+A_1,\ubk)  & (D_4+A_1,\ubk) &  (D_4(a_1),\ubk)  &  (A_2,\ubk)  \\
(D_6(a_2),11^\bk)  & (A_4+A_2,\ubk)  & (A_3+A_2+A_1,\ubk) &  (D_4(a_1),21^\bk) & (A_2,11^\bk)   \\
(D_5(a_1)+A_2,\ubk) & (A_4+2A_1,\ubk) &   (A_4,\ubk)    &   (D_4(a_1),111^\bk)  &  (3A_1,\ubk) \\
(A_5+A_1,\ubk) & (A_4+2A_1,11^\bk) &  (D_4(a_1)+A_1,\ubk) & (A_3+A_1,\ubk)  &  (2A_1,\ubk)
\end{array}
\]
The corresponding simple perverse sheaves $\IC(\cO,\cE)$ are direct summands of the Springer sheaf (which, of course, is induced from a maximal torus) by Lemma~\ref{lem:blocks-defect-0}. 

%To finish the proof of the claim, it suffices to show that the perverse sheaves corresponding to the remaining fourteen pairs belong to the triangulated subcategory generated by direct summands of induced sheaves. 
The remaining fourteen pairs (or more precisely the corresponding $W$-representations) are organized in four blocks as follows:
\begin{equation}\label{eqn:7-blocks}
\begin{aligned}
\mathsf{B}_1 &: (0,\ubk), \ (2A_2,11^\bk), \ (A_4+A_1,\ubk), \ (D_5(a_1)+A_1,\ubk) \\
\mathsf{B}_2 &: (4A_1,\ubk), \ (A_3+A_2,11^\bk), \ (A_5,\ubk) \\
\mathsf{B}_3 &: (A_1,\ubk), \ (A_2+A_1,11^\bk), \ (A_4,11^\bk), \ (E_6(a_3),11^\bk) \\
\mathsf{B}_4 &: (A_2+3A_1,\ubk), \ (A_3+A_2,\ubk), \ (A_4+A_1,11^\bk)
\end{aligned}
\end{equation}
(The blocks $\mathsf{B}_i$ also contain other $W$-representations corresponding to pairs not supported on $X$, which we have not listed.)
For each pair $(\cO,\cE)$ in this list, we will exhibit a direct summand of an induced perverse sheaf which is supported on $\overline{\cO}$ and whose restriction to $\cO$ is $\cE[\dim\cO]$; this will prove the claim.

The calculation relies on knowledge of inductions of $W$-representations over $\K$.
For instance, we have
\[
\mathrm{Ind}_{W(E_7)}^{W(E_8)} \chi_{1,0} \cong
\chi_{1,0} \oplus \chi_{35,2} \oplus \chi_{84,4} \oplus \chi_{8,1} \oplus \chi_{112,3}.
\]
By compatibility of the (ordinary) Springer correspondence with induction (which follows e.g.~from~\cite[Theorem~4.5]{genspring2}), we deduce that
\[
\Ind_{E_7}^{E_8}(\IC(\{0\},\underline{\K})) \cong
\IC(\{0\},\underline{\K}) \oplus \IC(2A_1, \underline{\K}) \oplus \IC(3A_1,\underline{\K}) \oplus \IC(A_1,\underline{\K}) \oplus \IC(A_2, \underline{\K}).
\]
Here, we write $\Ind_{E_7}^{E_8}$ for the functor $\Ind_{L\subset P}^G$ where $L$ is a Levi subgroup of type $E_7$ and $P$ is a parabolic subgroup containing $L$ as a Levi factor.
 
Observe that $\Ind_{E_7}^{E_8}(\IC(\{0\},\underline{\K}))$ is supported on $X$ and hence belongs to $\Perv_G^{\mathrm{Spr}}(\cN_G, \K)$ by Corollary~\ref{cor:image-springer}. Similarly for $\Ind_{E_7}^{E_8}(\IC(\{0\},\underline{\O}))$ and $\Ind_{E_7}^{E_8}(\bk\lotimes_\O  (\IC(\{0\},\underline{\O})))$; since induction commutes with modular reduction~\cite[Remark 2.23]{genspring1}, the latter is a modular reduction of $\Ind_{E_7}^{E_8}(\IC(\{0\},\underline{\K}))$. So it makes sense to apply block idempotents to all these perverse sheaves.

Consulting~\eqref{eqn:7-blocks}, we see that $e_3 \cdot \Ind_{E_7}^{E_8}(\IC(\{0\},\underline{\K}))=\IC(A_1, \underline{\K})$. Therefore $e_3 \cdot \Ind_{E_7}^{E_8}(\IC(\{0\},\underline{\O}))$ is an $\O$-form of $\IC(A_1, \underline{\K})$. Using~\cite[Proposition~2.8]{juteau} and Lemma~\ref{lem:compatibility}, it follows that $\overline{e}_3 \cdot \Ind_{E_7}^{E_8}(\bk\lotimes_\O  (\IC(\{0\},\underline{\O})))$ is supported on the closure of $A_1$ and its restriction to $A_1$ is the appropriate shift of $\ubk$. Thus, $\overline{e}_3 \cdot \Ind_{E_7}^{E_8}(\bk\lotimes_\O  (\IC(\{0\},\underline{\O})))$ is our desired direct summand of an induced perverse sheaf for the pair $(A_1,\ubk)$.

\begin{table}
\[
\begin{array}{|c|c|c|c|}
\hline
\text{Pair} & \K\text{-perverse sheaf on } \cN_{E_7} & \text{Idempotent} \\
\hline \hline
(\{0\}, \ubk) & \IC(\{0\},\underline{\K}) & e_1 \\
\hline
(A_1, \ubk) & \IC(\{0\},\underline{\K}) & e_3 \\
\hline
(2A_2,11^\bk) & \IC(2A_1, \underline{\K}) & e_1 \\
\hline
(A_2+A_1,11^\bk) & \IC(2A_1, \underline{\K}) & e_3 \\
\hline
(A_3+A_2,11^\bk) & \IC(3A_1', \underline{\K}) & e_2 \\
\hline
(A_3+A_2,\ubk) & \IC(3A_1', \underline{\K}) & e_4 \\
\hline
(4A_1,\ubk) & \IC(4A_1, \underline{\K}) & e_2 \\
\hline
(A_2+3A_1,\ubk) & \IC(4A_1, \underline{\K}) & e_4 \\
\hline
(A_4+A_1,\ubk) & \IC(A_2+A_1, \underline{\K}) & e_1 \\
\hline
(A_4, 11^\bk) & \IC(A_2+A_1, \underline{\K}) & e_3 \\
\hline
(A_4+A_1,11^\bk) & \IC(A_2+A_1, \underline{\K}) & e_4 \\
\hline 
(A_5,\ubk) & \IC(2A_2, \underline{\K}) & e_2 \\
\hline
(D_5(a_1)+A_1,\ubk) & \IC(A_3, \underline{\K}) & e_1 \\
\hline
(E_6(a_3),11^\bk) & \IC(A_3, \underline{\K}) & e_3 \\
\hline
\end{array}
\]
\caption{Calculations for type $E_8$ and $\ell = 7$}\label{tab:e8-clean}
\end{table}

%\tablehead{\hline &1 &2A_1 &A_2{+}A_1 2&A_3 &4A_1 &3A_1' &2A_2 \\\hline}
%\tabletail{\hline}
%\begin{supertabular}{|R|RRRRRRR|}
%\shrinkheight{30pt}

%\begin{table}
%\[
%\begin{array}{|c|ccccccc|}
%\hline
%E_8 \backslash E_7&1 &2A_1 &(A_2{+}A_1, 2)&A_3 &4A_1 &3A_1' &2A_2 \\
%\hline
%0 &1&.&.&.&.&.&.\\
%(2A_2, 11)&.&1&.&.&.&.&1\\
%(A_4{+}A_1, 2)&.&.&1&1&.&.&1\\
%D_5(a_1){+}A_1 &.&.&.&1&.&.&.\\
%\hline
%4A_1 &.&.&.&.&1&1&.\\
%(A_3{+}A_2, 11)&.&.&.&.&.&1&1\\
%A_5 &.&.&.&.&.&.&1\\
%\hline
%A_1 &1&.&.&.&.&.&.\\
%(A_2{+}A_1, 11)&.&1&.&.&.&.&.\\
%(A_4, 11)&.&.&1&1&.&.&.\\
%(E_6(a_3), 11)&.&.&.&1&.&.&.\\
%\hline
%A_2{+}3A_1 &.&.&.&.&1&1&.\\
%(A_3{+}A_2, 2)&.&.&1&1&.&1&1\\
%(A_4{+}A_1, 11)&.&.&1&1&.&.&1\\
%\hline
%\end{array}
%\]
%\caption{Calculations for type $E_8$ and $\ell = 7$}\label{tab:e8-clean}
%\end{table}

%\begin{table}
%\[
%\begin{array}{c}
%\begin{array}{|c|c|cccc|}
%\hline
%e_1\Ind_{E_7}^{E_8}&e_3\Ind_{E_7}^{E_8}&0 &2A_1 &(A_2{+}A_1, 2)&A_3  \\
%\hline
%0 &A_1 &1&.&.&.\\
%(2A_2, 11)&(A_2{+}A_1, 11)&.&1&.&.\\
%(A_4{+}A_1, 2)&(A_4, 11)&.&.&1&1\\
%D_5(a_1){+}A_1 &(E_6(a_3), 11)&.&.&.&1\\
%\hline
%\end{array}
%\\
%\\
%\begin{array}{|c|c|ccc|}
%\hline
%e_2\Ind_{E_7}^{E_8}&e_4\Ind_{E_7}^{E_8}&  4A_1 &3A_1' &2A_2\\
%\hline
%4A_1 & A_2{+}3A_1 &1&1&.\\
%(A_3{+}A_2, 11)&(A_3{+}A_2, 2)&.&1&1\\
%A_5 &(A_4{+}A_1, 11)&.&.&1\\
%\hline
%\end{array}
%\end{array}
%\]
%\caption{Calculations for type $E_8$ and $\ell = 7$}\label{tab:e8-clean}
%\end{table}

Similar calculations apply to each of the other pairs appearing in~\eqref{eqn:7-blocks}.  For each such pair, Table~\ref{tab:e8-clean} lists a $\K$-perverse sheaf on $\cN_{E_7}$ and a block idempotent that can be used to produce an appropriate direct summand of an induced perverse sheaf.

%--------------------------------------------------------------------
\subsection{Proof of Theorem~\ref{thm:clean} for $G$ of type $B_4$ and $\ell = 3$}
\label{ss:cleanness-B4}
%--------------------------------------------------------------------

In this proof, we will denote nilpotent orbits by partitions of $9$.  Let $\cE_{531}$ be the local system on the orbit $531$ described in~\cite[Theorem~8.3]{genspring2}.  Then $(531,\cE_{531})$ is the unique $0$-cuspidal pair in $\fN_{G,\bk}$. Since this cuspidal pair has trivial central character, we need only consider pairs with trivial central character; equivalently, we can and will redefine $G$ to be the adjoint group of type $B_4$, namely $\SO(9)$.

Let $X$ be the union of all nilpotent orbits for $G$ that are contained in the closure of $531$ but different from it.  It turns out that for all $x \in X$, the group $A_G(x)$ is either trivial or of order~$2$.  When it is of order~$2$, we will denote the nontrivial irreducible local system on the orbit of $x$ by $\varepsilon^\bk$ or $\varepsilon^\K$.  The trivial local system will still be denoted by $\ubk$ or $\underline{\K}$.  The orbit $441$ is open in $X$; let $X'$ be its complement in $X$.

As in the preceding subsection, it suffices to prove the following claim: \textit{The category $\Db_G(X,\bk)$ is generated as a triangulated category by direct summands of perverse sheaves induced from proper Levi subgroups.} As in~\S\ref{ss:cleanness-E8}, to prove this claim we will exhibit, for any pair $(\cO,\cE)$ with $\cO \subset X$, a direct summand of a perverse sheaf induced from a proper Levi subgroup which is supported on $\overline{\cO}$ and whose restriction to $\cO$ is $\cE[\dim \cO]$.

There are fourteen pairs $(\cO,\cE)$ with $\cO \subset X$, thirteen of which actually satisfy $\cO \subset X'$. The modular Springer correspondence in this case is known from~\cite{jls}. It turns out that all thirteen pairs supported on $X'$ are in the principal series $\fN^{(T,\{0\},\underline{\E})}_{G,\E}$ for both $\E = \K$ and $\E = \bk$.  

The pair $(441,\ubk)$ is \emph{not} in the principal series $\fN^{(T,\{0\},\ubk)}_{G,\bk}$.  In fact, this pair must belong to the series $\fN^{(A_2,\cO_\reg,\ubk)}_{(G,\bk)}$, as that is the only non-principal, non-cuspidal induction series that has trivial central character (see~\cite{genspring2}).  The Levi subgroup $A_2$ belongs to the $3$-Sylow class of $G$ (see~\cite[\mglsylow]{genspring3}), and so by~\cite[\mgsemisimple]{genspring3} the induced perverse sheaf $\Ind_{A_2}^{B_4}(\IC(\cO_\reg,\ubk))$ is semisimple.  In particular, $\IC(441,\ubk)$ is a direct summand of $\Ind_{A_2}^{B_4}(\IC(\cO_\reg,\ubk))$, which takes care of the pair $(441,\ubk)$.

We now turn our attention to the thirteen pairs on $X'$, with a strategy very much like that employed in the preceding subsection.  Six of these thirteen pairs correspond to $\bk[W]$-representations in blocks of defect~$0$:
\[
\begin{array}{cccccc}
(51111, \ubk) &
(51111, \varepsilon^\bk) &
(333, \ubk) &
(33111, \varepsilon^\bk) &
(32211, \ubk) &
(2211111, \ubk)
\end{array}
\]
The corresponding perverse sheaves are direct summands of the Springer sheaf by Lemma~\ref{lem:blocks-defect-0}.  The remaining seven pairs are organized in four $3$-blocks:
\[
\begin{aligned}
\mathsf{B}_1 &: (111111111, \ubk), \ (22221, \ubk) \\
\mathsf{B}_2 &: (3111111,\ubk), \ (33111, \ubk) \\
\mathsf{B}_3 &: (32211, \varepsilon^\bk), \ (522, \ubk) \\
\mathsf{B}_4 &: (3111111, \varepsilon^\bk)
\end{aligned}
\]
(Again, we have omitted pairs in these blocks that are not supported on $X'$.)  For each pair $(\cO,\cE)$ in this list, there exists a direct summand of a perverse sheaf induced from a Levi subgroup of type $B_3$
%induced from some Levi subgroup 
that is supported on $\overline{\cO}$ and whose restriction to $\cO$ is isomorphic to $\cE[\dim \cO]$.  
%In fact,  a complete set of such objects can be obtained by induction from .  
Table~\ref{tab:b4-clean} gives a list of $\K$-perverse sheaves on $\cN_{B_3}$ and block idempotents that can be used to produce these objects.

\begin{table}
\[
\begin{array}{|c|c|c|}
\hline
\text{Pair} & \K\text{-perverse sheaf on } \cN_{B_3} & \text{Idempotent} \\
\hline \hline
(111111111, \ubk) & \IC(1111111, \underline{\K}) & e_1 \\
\hline
(22221, \ubk) & \IC(22111,\underline{\K}) & e_1 \\
\hline
(3111111,\ubk) & \IC(1111111, \underline{\K}) & e_2 \\
\hline
(33111, \ubk) & \IC(22111,\underline{\K}) & e_2 \\
\hline
(32211, \varepsilon^\bk) & \IC(31111, \varepsilon^\K) & e_3 \\
\hline
(522, \ubk) & \IC(322, \underline{\K}) & e_3 \\
\hline
(3111111, \varepsilon^\bk) & \IC(31111, \varepsilon^\K) & e_4 \\
\hline
\end{array}
\]
\caption{Calculations for type $B_4$ and $\ell = 3$}\label{tab:b4-clean}
\end{table}

%%%%%%%%%%%%%%%%%%%%%%%%%%%%%%%%%%%%%%%%%%%%%%%%%%%%%%%%%%%%%%%%%%%%%%%%%%%
\section{Consequences of the cleanness conjecture}
\label{sec:consequences}
%%%%%%%%%%%%%%%%%%%%%%%%%%%%%%%%%%%%%%%%%%%%%%%%%%%%%%%%%%%%%%%%%%%%%%%%%%%

In this section, we prove that the cleanness conjecture implies that $\Perv_G(\cN_G,\bk)$ and $\Db_G(\cN_G,\bk)$ admit direct-sum decompositions indexed by $0$-cuspidal data.  Along the way, we study the behaviour of the induction and restriction functors with respect to $0$-series, and we prove that the notions of `supercuspidal' and `$0$-cuspidal' coincide (assuming the cleanness conjecture). In Section~\ref{ss:semisimplicity}, we give a semisimplicity criterion for $\Perv_G(\cN_G,\bk)$. Apart from that final section, we continue to assume that $\ell$ is rather good for $G$. Note that we do not make Conjecture~\ref{conj:clean} a blanket assumption, but rather we will assume it (or its equivalent form~\eqref{eqn:clean-restate}) where it is needed.

%\begin{rmk}
%Most results in this section include the assumption that Conjecture~\ref{conj:clean} holds for $G$ and all its Levi subgroups. Note that this condition is satisfied at least under any of the conditions appearing in Theorem~\ref{thm:clean}. Indeed, if $\ell \nmid |W|$, then $\ell$ does not divide the order of the Weyl group of any Levi subgroup of $G$ either. And if $G/Z(G)$ is semisimple of type $A$, or if $G$ is quasi-simple and simply connected of type $C_3$ or of exceptional type, then any Levi subgroup $L$ of $G$ supporting a $0$-cuspidal pair and different from $G$ is of type $A$. In any case, Theorem~\ref{thm:clean} also applies to Levi subgroups of $G$.
%\end{rmk} 
% Analogous remark now made in the introduction. -- AH

%--------------------------------------------------------------------
\subsection{$0$-cuspidal pairs and projectivity}
%--------------------------------------------------------------------

\begin{lem}\label{lem:dist-semis}
Assume that $G$ is semisimple.
If $\cO$ is a distinguished nilpotent orbit, then $\Db_G(\cO,\bk)$ is semisimple.  That is, for any two irreducible $G$-equivariant local systems $\cE, \cE'$ on $\cO$, we have
\[
\Hom^k_{\Db_G(\cO,\bk)}(\cE,\cE') \cong
\begin{cases}
\bk & \text{if $\cE \cong \cE'$ and $k = 0$,} \\
0 & \text{otherwise.}
\end{cases}
\]
\end{lem}
\begin{proof}
Choose a point $x \in \cO$, and consider its stabilizer $G_x$.  Since $\cO \cong G/G_x$, we have an equivalence of categories $\Db_G(\cO, \bk) \cong \Db_{G_x}(\pt, \bk)$.  Since $\cO$ is distinguished and $G$ is semisimple, the identity component $G_x^\circ$ is a unipotent group. By~\cite[Theorem~6.1]{mostow}, there exists a subgroup $H \subset G_x$ such that multiplication induces an isomorphism $G_x \cong H \ltimes G_x^\circ$; in particular, this implies that $H \cong A_G(x)$. By~\cite[Theorem~3.7.3]{bl}, restriction induces a fully faithful functor $D_{G_x}(\pt, \bk) \to D_H(\pt, \bk)$.
Since $H$ is a finite group, the equivariant derived category $D_H(\pt, \bk)$ is just the derived category of $H$-representations, see~\cite[Theorem~8.3.1]{bl}.  Since $\ell$ does not divide the order of $H \cong A_G(x)$, the category of $H$-representations is semisimple, which implies our claim.
\end{proof}

\begin{prop}\label{prop:zcusp-vanish}
%Assume that 
%%$\ell$ is rather cogood for $G$, and that 
%Conjecture~{\rm \ref{conj:clean} is true}. 
Let $(\cO,\cE) \in \fN^\zcusp_{G,\bk}$, and assume that~\eqref{eqn:clean-restate} holds for $(\cO,\cE)$. Then for all $(\cO',\cE') \in \fN_{G,\bk}$ we have
\begin{equation*}
\Ext^1_{\Perv_G(\cN_G,\bk)}(\IC(\cO,\cE),\IC(\cO',\cE')) = \Ext^1_{\Perv_G(\cN_G,\bk)}(\IC(\cO',\cE'),\IC(\cO,\cE))=0.
\end{equation*}
%\begin{equation}\label{eqn:zcusp-vanish}
%\begin{cases}
%\Hom^k(\IC(\cO,\cE),\IC(\cO',\cE')) = 0 \\
%\Hom^k(\IC(\cO',\cE'), \IC(\cO,\cE)) = 0
%\end{cases}
%\qquad\text{for all $k \neq 0$.}
%\end{equation}
In particular, $\IC(\cO,\cE)$ is a projective and injective object of $\Perv_G(\cN_G,\bk)$.
\end{prop}
\begin{proof}
We prove the vanishing of $\Ext^1_{\Perv_G(\cN_G,\bk)}(\IC(\cO,\cE),\IC(\cO',\cE'))$; the other statement can be proved by similar arguments. As noted in the proof of Lemma~\ref{lem:cogood-reduction}, we have an equivalence $\Perv_{G}(\cN_G,\bk) \cong \Perv_{G/Z(G)^\circ}(\cN_{G/Z(G)^\circ},\bk)$, so we can assume that $G$ is semisimple. In this case, we will prove that
\begin{equation}
\label{eqn:zcusp-vanish}
\Hom^k_{\Db_G(\cN_G,\bk)}(\IC(\cO,\cE),\IC(\cO',\cE')) = 0 \qquad \text{for $k>0$.}
\end{equation}
By~\cite[Remarque~3.1.17(ii)]{bbd}, this will imply the proposition. Note that, by~\cite[Proposition~2.6]{genspring2}, $\cO$ is a distinguished nilpotent orbit. For brevity, we write simply $\Hom^k$ instead of $\Hom^k_{\Db_G(\cN_G,\bk)}$.

If $(\cO',\cE') = (\cO,\cE)$, then by~\eqref{eqn:clean-restate} and adjunction, we have $\Hom^k(\IC(\cO,\cE), \IC(\cO,\cE)) \cong \Hom^k(\cE,\cE)$.  The latter vanishes for $k > 0$ by Lemma~\ref{lem:dist-semis}.

Assume henceforth that $(\cO',\cE') \neq (\cO,\cE)$.  We proceed by downward induction (with respect to $\preceq_G$) on the induction series to which $(\cO',\cE')$ belongs.  The base case is that in which $(\cO',\cE')$ is cuspidal. Since $(\cO,\cE)$ is cuspidal (see Lemma~\ref{lem:comparison-cusp}), $(\cO,\cE)$ and $(\cO',\cE')$ have distinct central characters by Proposition~\ref{prop:unicity-cuspidal}.  Then Lemma~\ref{lem:centralchar-decomp} implies that
$\Hom^k(\IC(\cO,\cE),\IC(\cO',\cE')) 
%= \Hom^k(\IC(\cO',\cE'), \IC(\cO,\cE)) 
= 0$
for all $k$, as desired.

Next, consider a series $\fN^{(L,\cO_L,\cE_L)}_{G,\bk}$ with $L \neq G$.  Let $\cF$ denote the head of $\Ind_{L\subset P}^G (\IC(\cO_L,\cE_L))$, and let $\cF'$ denote the kernel of the surjection $\Ind_{L \subset P}^G (\IC(\cO_L,\cE_L)) \twoheadrightarrow \cF$.  Using~\eqref{eqn:clean-restate} and adjunction, we have
\[
\Hom^k(\IC(\cO,\cE), \cF) \cong \Hom^k(\cE[\dim \cO], j_\cO^!\cF),
\]
where as above $j_\cO: \cO \hookrightarrow \cN$ is the inclusion map. By Lemma~\ref{lem:dist-semis}, the latter $\Hom$-group vanishes for all but finitely many values of $k$. Suppose that it is not always zero, and let $m$ be the largest integer such that $\Hom^m(\IC(\cO,\cE), \cF) \ne 0$. Of course, $\cF$ is a semisimple object that does not contain $\IC(\cO,\cE)$ as a summand (by Lemma~\ref{lem:comparison-cusp}), so $\Hom^0(\IC(\cO,\cE), \cF) = 0$.  In other words, we must have $m > 0$.

By Lemma~\ref{lem:ind-series-po}, each composition factor $\cG$ of $\cF'$ is either in the series associated to $(L,\cO_L,\cE_L)$---in which case it is isomorphic to a direct summand of $\cF$---or in a larger series, for which the conclusion of the proposition is already known to hold. In either case, we have that $\Hom^{m+1}(\IC(\cO,\cE),\cG) = 0$, and hence $\Hom^{m+1}(\IC(\cO,\cE),\cF') = 0$. Now consider the long exact sequence
%\begin{multline*}
\[
\cdots \to \Hom^m(\IC(\cO,\cE), \Ind_{L \subset P}^G (\IC(\cO_L,\cE_L))) \to \Hom^m(\IC(\cO,\cE), \cF)
\to \Hom^{m+1}(\IC(\cO,\cE), \cF') \to \cdots
\]
%\end{multline*}
The first term vanishes by adjunction and Lemma~\ref{lem:comparison-cusp}, and we have already seen that the last term vanishes. But then the middle term vanishes as well, contradicting our assumption.

Thus, $\Hom^k(\IC(\cO,\cE),\cF) = 0$ for all $k > 0$.  Objects of the form $\IC(\cO',\cE')$ with $(\cO',\cE') \in \fN^{(L,\cO_L,\cE_L)}_{G,\bk}$ are precisely the direct summands of $\cF$, so we conclude that $\Hom^k(\IC(\cO,\cE),\IC(\cO',\cE')) = 0$ for all $k > 0$ and all $(\cO',\cE') \in \fN^{(L,\cO_L,\cE_L)}_{G,\bk}$.
%A similar argument shows the vanishing of the other $\Hom$-group in~\eqref{eqn:zcusp-vanish}.
%
%The last assertion in the lemma follows from the $k=1$ case of~\eqref{eqn:zcusp-vanish}, together with~\cite[Remarque~3.1.17(ii)]{bbd}.
\end{proof}

%--------------------------------------------------------------------
\subsection{Decomposition of $\Perv_G(\cN_G, \bk)$ according to induction $0$-series}
%--------------------------------------------------------------------

Note that, under our assumption that $\ell$ is rather good for $G$, the category of $G$-equivariant $\bk$-local systems on $\cO$ is semisimple for any $G$-orbit $\cO \subset \cN_G$. Using~\cite[Remark~3 after Theorem~3.2.1]{bgs}, it follows that the category $\Perv_G(\cN_G,\bk)$ has enough projective objects.
%and enough injective objects. 
By standard arguments this implies that $\Perv_G(\cN_G,\bk)$ is equivalent to the category of finite dimensional modules over some finite dimensional $\bk$-algebra; in particular,
it makes sense to consider projective covers and injective hulls in this category.

\begin{prop}\label{prop:projective}
Assume that 
%$\ell$ is rather cogood for $G$, and that 
Conjecture~{\rm \ref{conj:clean}} is true for all Levi subgroups of $G$. If $(\cO,\cE) \in \fN^{\zs(L,\cO_L,\cE_L)}_{G,\bk}$, then every composition factor of the projective cover or the injective hull of $\IC(\cO,\cE)$ is of the form $\IC(\cO',\cE')$ for some $(\cO',\cE') \in \fN^{\zs(L,\cO_L,\cE_L)}_{G,\bk}$.
\end{prop}

\begin{proof}
We proceed by induction on the semisimple rank of $G$.
%If $G$ is a torus, then $\fN_{G,\bk}$ is a singleton, and there is nothing to prove.  Otherwise, 
We can assume $G$ is a semisimple group of the form described in Lemma~\ref{lem:cogood-reduction}.  If $G$ is a product of proper subgroups, then the proposition holds for its factors by induction, and hence for $G$.  Thus, it suffices to treat the cases where $G$ is semisimple of type $A$, or quasi-simple, simply connected, and not of type $A$. 

If $G$ is semisimple of type $A$, then the direct sum decomposition in Proposition~\ref{prop:0-series-A}\eqref{it:0-ser-sum} implies the result.

Assume henceforth that $G$ is quasi-simple, simply connected, and not of type $A$.  We proceed by downward induction with respect to $\preceq_G$ on the triple $(L,\cO_L,\cE_L)$. If $L = G$, so that $(\cO,\cE) \in \fN^\zcusp_{G,\bk}$, then we saw in Proposition~\ref{prop:zcusp-vanish} that $\IC(\cO,\cE)$ is its own projective cover and injective hull. Otherwise, suppose the proposition has already been established for all $(L', \cO_{L'}, \cE_{L'}) \succ_G (L,\cO_L,\cE_L)$.  Let $(\cO,\cE) \in \fN^{\zs(L,\cO_L,\cE_L)}_{G,\bk}$.  Here, $L \neq G$, so by Proposition~\ref{prop:simply-conn}, $(\cO,\cE)$ cannot be cuspidal.

We first claim that the projective cover $\mathcal{P}$ of $\IC(\cO,\cE)$ cannot contain any composition factor lying in a $0$-series $\fN^{\zs(L', \cO_{L'}, \cE_{L'})}_{G,\bk}$ with $(L', \cO_{L'}, \cE_{L'}) \succ_G (L,\cO_L,\cE_L)$.  Indeed, if it contained such a composition factor, say $\IC(\cO',\cE')$, then there would be a nonzero map $\mathcal{P} \to \mathcal{I}'$, where $\mathcal{I}'$ is the injective envelope of $\IC(\cO',\cE')$. But the existence of such a map would imply that $\IC(\cO,\cE)$ occurs as a composition factor in $\mathcal{I}'$, contradicting the fact (known inductively) that all its composition factors lie in $\fN^{\zs(L', \cO_{L'}, \cE_{L'})}_{G,\bk}$.

Since $(\cO,\cE)$ is not cuspidal, it lies in some induction series $\fN^{(M,\cO_M,\cE_M)}_{G,\bk}$ with $M \neq G$. By Theorem~\ref{thm:0-series}, replacing $(M,\cO_M,\cE_M)$ by a $G$-conjugate if necessary, we can assume that $L \subset M$ and that $(\cO_M,\cE_M)$ lies in $\fN^{\zs(L,\cO_L,\cE_L)}_{M,\bk}$. Let $\cF$ be the projective cover of $\IC(\cO_M,\cE_M)$ in $\Perv_M(\cN_M,\bk)$.  By induction, all composition factors of $\cF$ lie in $\fN^{\zs(L,\cO_L,\cE_L)}_{M,\bk}$.  Since $\Ind_{M \subset Q}^G$ has an exact right adjoint, it takes projectives to projectives.  Since $\Ind_{M \subset Q}^G$ is itself exact,
% and kills no nonzero object (see~\cite[Corollary~2.15]{genspring1}), 
$\Ind_{M \subset Q}^G(\cF)$ is a projective perverse sheaf that has $\IC(\cO,\cE)$ as a quotient.  In particular, the projective cover $\mathcal{P}$ of $\IC(\cO,\cE)$ occurs as a direct summand of $\Ind_{M \subset Q}^G(\cF)$. By Corollary~\ref{cor:ind-0-series-po}, this implies that all composition factors of $\mathcal{P}$ lie in induction $0$-series $\fN^{\zs(L', \cO_{L'}, \cE_{L'})}_{G,\bk}$ with $(L', \cO_{L'}, \cE_{L'}) \succeq_G (L,\cO_L,\cE_L)$.  But we saw above that $(L', \cO_{L'}, \cE_{L'}) \succ_G (L,\cO_L,\cE_L)$ cannot occur, so all composition factors of $\mathcal{P}$ must indeed lie in $\fN^{\zs(L,\cO_L,\cE_L)}_{G,\bk}$.

The argument for injective hulls is similar.
\end{proof}

\begin{thm}\label{thm:zseries-decomp}
Assume that 
%$\ell$ is rather cogood for $G$, and that 
Conjecture~{\rm \ref{conj:clean}} is true for all Levi subgroups of $G$. Then we have
\[
\Perv_G(\cN_G,\bk) \cong
\bigoplus_{(L,\cO_L, \cE_L) \in \fM^0_{G,\bk}} \Serre(\fN^{\zs(L,\cO_L,\cE_L)}_{G,\bk}),
\]
where $\Serre(\fN^{\zs(L,\cO_L,\cE_L)}_{G,\bk})$ denotes the Serre subcategory generated by the simple perverse sheaves $\IC(\cO,\cE)$ with $(\cO,\cE) \in \fN^{\zs(L,\cO_L,\cE_L)}_{G,\bk}$.
\end{thm}

\begin{proof}
Such a decomposition follows from the claim that
\[
\Ext^1_{\Perv_G(\cN_G,\bk)}(\IC(\cO,\cE), \IC(\cO',\cE')) = 0
\]
whenever $(\cO,\cE)$ and $(\cO',\cE')$ belong to different $0$-series.  This $\Ext^1$-vanishing is immediate from the information on projectives (or injectives) provided by Proposition~\ref{prop:projective}.
\end{proof}

%%%%%%%%%%%%%%%%%%%%%%%%%%%%%%%%%%%%%%%%%%%%%%%%%%%%%%%%%%%%%%%%%%%%%%%%%%%
%\section{Consequences on the structure of $\Db_G(\cN_G,\bk)$}
%\label{ss:clean-consequences-D}
%%%%%%%%%%%%%%%%%%%%%%%%%%%%%%%%%%%%%%%%%%%%%%%%%%%%%%%%%%%%%%%%%%%%%%%%%%%

%--------------------------------------------------------------------
\subsection{Some consequences of the decomposition}
%--------------------------------------------------------------------

We can now refine Corollary~\ref{cor:ind-0-series-po}.

\begin{cor}
\label{cor:ind-0-series-po-new}
Assume that Conjecture~{\rm \ref{conj:clean}} is true for all Levi subgroups of $G$.
Let $L \subset M \subset G$ be Levi subgroups, and let $P \subset Q \subset G$ be corresponding parabolic subgroups.  
\begin{enumerate}
\item 
\label{it:ind-0}
If $(\cO,\cE) \in \fN^{\zs(L,\cO_L,\cE_L)}_{M,\bk}$, then every composition factor of $\Ind_{M\subset Q}^G (\IC(\cO,\cE))$ lies in $\fN^{\zs(L,\cO_L,\cE_L)}_{G,\bk}$.
\item 
\label{it:res-0}
If $(\cO,\cE) \in \fN^{\zs(L,\cO_L,\cE_L)}_{G,\bk}$, then every composition factor of $\Res_{M\subset Q}^G (\IC(\cO,\cE))$ lies in $\fN^{\zs(L,\cO_L,\cE_L)}_{M,\bk}$.
\end{enumerate}
\end{cor}

\begin{proof}
\eqref{it:ind-0}
Suppose $(\cO,\cE)$ belongs to the series $\fN_{M,\bk}^{(N,\cO_N,\cE_N)}$.  From Theorem~\ref{thm:0-series}, we know that $(L,\cO_L,\cE_L) \preceq_M (N,\cO_N,\cE_N)$, so we may assume that $L \subset N \subset M$.  Let $R \subset G$ be a parabolic subgroup with Levi factor $N$, and such that $P \subset R \subset Q$.  Our assumptions imply that $\IC(\cO,\cE)$ is a quotient of $\Ind_{N \subset R \cap M}^M \IC(\cO_N,\cE_N)$.

Now, let $(L',\cO_{L'}, \cE_{L'})$ be a $0$-cuspidal datum for $G$ such that $\Ind_{M \subset Q}^G \IC(\cO,\cE)$ contains a composition factor in the Serre subcategory $\Serre(\fN^{\zs(L',\cO_{L'},\cE_{L'})}_{G,\bk})$.  By Theorem~\ref{thm:zseries-decomp}, it must also have a simple quotient, say $\IC(\cO',\cE')$, belonging to $\Serre(\fN^{\zs(L',\cO_{L'},\cE_{L'})}_{G,\bk})$.  By exactness and transitivity of induction (see~\cite[Lemma~2.6]{genspring1}), $\IC(\cO',\cE')$ is also a quotient of $\Ind_{N \subset R}^G \IC(\cO_N,\cE_N)$.  In other words, we have
\[
(\cO',\cE') \in \fN_{G,\bk}^{(N,\cO_N,\cE_N)} \cap \fN_{G,\bk}^{\zs(L',\cO_{L'},\cE_{L'})}.
\]
Theorem~\ref{thm:0-series} tells us that $(L',\cO_{L'}, \cE_{L'})$ is the unique maximal $0$-cuspidal datum such that $(L',\cO_{L'},\cE_{L'}) \preceq_G (N,\cO_N,\cE_N)$.  In particular, the $0$-cuspidal datum $(L',\cO_{L'}, \cE_{L'})$ is uniquely determined; all composition factors of $\Ind_{M \subset Q}^G \IC(\cO,\cE)$ belong to the same $0$-series.  

Thus, to finish the proof, it suffices to show that $\Ind_{M \subset Q}^G \IC(\cO,\cE)$ has at least one composition factor in $\fN_{G,\bk}^{\zs(L,\cO_L,\cE_L)}$.  Let $\cE^\K$ be the $\K$-local system on $\cO$ such that $\theta_M(\cO, \cE^\K) = (\cO,\cE)$.  Let $\cO''$ be the induced nilpotent orbit $\mathrm{Ind}_M^G \cO$.  Recall that $\Ind_{M \subset Q}^G \IC(\cO,\cE)$ and $\Ind_{M \subset Q}^G \IC(\cO,\cE^\K)$ are both supported on $\overline{\cO''}$, so that the complexes
\[
\cL := \Ind_{M \subset Q}^G \IC(\cO,\cE)_{|\cO''}[-\dim \cO'']
\qquad\text{and}\qquad
\cL^\K := \Ind_{M \subset Q}^G \IC(\cO,\cE^\K)_{|\cO''}[-\dim \cO''],
\]
are in fact both local systems.  It is easy to see from the proof of~\cite[Corollary~2.15]{genspring1} that these local systems are nonzero, and that $\cL$ is none other than the modular reduction of $\cL^\K$.

Both $\cL$ and $\cL^\K$ are semisimple since $\ell$ is rather good. If $\cE''$ is a simple direct summand of $\cL$, then there is a corresponding direct summand $\cE^\K{}''$ of $\cL^\K$, with $\theta_G(\cO'', \cE^\K{}'') = (\cO'',\cE'')$.  The $\K$-version of the corollary (which is obvious by semisimplicity of all perverse sheaves under consideration) implies that $(\cO'',\cE'') \in \fN_{G,\bk}^{\zs(L,\cO_L,\cE_L)}$.  Since $\IC(\cO'',\cE'')$ is a composition factor of $\Ind_{M \subset Q}^G \IC(\cO,\cE)$, we are done.

\eqref{it:res-0} If $\Res_{M\subset Q}^G (\IC(\cO,\cE))$ has a composition factor in $\fN^{\zs(L',\cO_{L'},\cE_{L'})}_{M,\bk}$, then Theorem~\ref{thm:zseries-decomp} implies that it has a simple subobject $\IC(\cO',\cE')$ in $\fN^{\zs(L',\cO_{L'},\cE_{L'})}_{M,\bk}$.  By adjunction (see~\cite[\S 2.1]{genspring1}), there is a nonzero map
\[
\Ind_{M \subset Q}^G \IC(\cO',\cE') \to \IC(\cO,\cE),
\]
which must be surjective since its codomain is simple.
Part~\eqref{it:ind-0} of the corollary tells us that the first term must lie in $\Serre(\fN^{\zs(L',\cO_{L'},\cE_{L'})}_{G,\bk})$, so that $(L,\cO_L,\cE_L)$ and  $(L',\cO_{L'},\cE_{L'})$ must be $G$-conjugate.  Finally, Corollary~\ref{cor:order-0-Levi} tells us that these two $0$-cuspidal data are also $M$-conjugate.  Thus, every composition factor of $\Res_{M\subset Q}^G (\IC(\cO,\cE))$ lies in $\fN^{\zs(L,\cO_L,\cE_L)}_{M,\bk}$, as desired.
\end{proof}

\begin{prop}
\label{prop:0cusp-supercusp}
Assume that
Conjecture~{\rm \ref{conj:clean}} is true for all Levi subgroups of $G$.
\begin{enumerate}
\item 
\label{it:supercusp-0cusp}
A pair $(\cO,\cE) \in \fN_{G,\bk}$ is supercuspidal if and only if it is $0$-cuspidal.
\item 
\label{it:supercusp-0cusp-series}
Let $(L,\cO_L,\cE_L)$ be a supercuspidal datum. Then we have
\[
\fN^{\super(L,\cO_L,\cE_L)}_{G,\bk} = \fN^{\zs(L,\cO_L,\cE_L)}_{G,\bk}.
\]
\end{enumerate}
\end{prop}

\begin{proof}
\eqref{it:supercusp-0cusp}
By Lemma~\ref{lem:comparison-cusp}, a supercuspidal pair is necessarily $0$-cuspidal. On the other hand, if $(\cO,\cE)$ is not supercuspidal, then it is a composition factor of a perverse sheaf of the form $\Ind_{L \subset P}^G(\IC(\cO_L,\cE_L))$ with $L \neq G$. The pair $(\cO_L,\cE_L) \in \fN_{L,\bk}$ belongs to some induction $0$-series $\fN_{L,\bk}^{\zs(M,\cO_M,\cE_M)}$ with $M \subset L$. Then by Corollary~\ref{cor:ind-0-series-po-new}\eqref{it:ind-0} we have $(\cO,\cE) \in \fN_{G,\bk}^{\zs(M,\cO_M,\cE_M)}$, with $M \neq G$, so that this pair is not $0$-cuspidal.

\eqref{it:supercusp-0cusp-series}
The inclusion $\fN^{\zs(L,\cO_L,\cE_L)}_{G,\bk} \subset \fN^{\super(L,\cO_L,\cE_L)}_{G,\bk}$ was proved in
Corollary~\ref{cor:comparison-0series-superseries}. The other inclusion follows from Corollary~\ref{cor:ind-0-series-po-new}\eqref{it:ind-0} (applied to $M=L$).
\end{proof}

%--------------------------------------------------------------------
\subsection{Decomposition of $\Db_G(\cN_G, \bk)$ according to induction $0$-series}
%--------------------------------------------------------------------

\begin{thm}
\label{thm:decomposition-D-0series}
Assume that 
Conjecture~{\rm \ref{conj:clean}} is true for all Levi subgroups of $G$. Then we have
\[
\Db_G(\cN_G,\bk) \cong
\bigoplus_{(L,\cO_L,\cE_L) \in \fM^0_{G,\bk}} \Db_G(\cN_G,\bk)^{\zs(L,\cO_L,\cE_L)},
\]
where $\Db_G(\cN_G,\bk)^{\zs(L,\cO_L,\cE_L)}$ is the full triangulated subcategory of $\Db_G(\cN_G,\bk)$ generated by the perverse sheaves $\IC(\cO,\cE)$ with $(\cO,\cE) \in \fN_{G,\bk}^{\zs(L,\cO_L,\cE_L)}$.
\end{thm}

For the proof of the theorem we will need the following lemma, whose proof (which does not require Conjecture~{\rm \ref{conj:clean}}) is a variant of the proof of~\cite[\mgdisjoint]{genspring3}.

\begin{lem}
\label{lem:Hom-vanishing-Ind}
If $(L,\cO_L,\cE_L)$ and $(M,\cO_M,\cE_M)$ are non-conjugate cuspidal data, and if $P,Q \subset G$ are parabolic subgroup with respective Levi subgroups $L$ and $M$, then we have
\[
\Hom^k_{\Db_G(\cN_G,\bk)}(\Ind_{L \subset P}^G(\IC(\cO_L,\cE_L)), \Ind_{M \subset Q}^G(\IC(\cO_M,\cE_M)))=0
\]
for any $k \in \Z$.
\end{lem}

\begin{proof}
Using the fact that induction commutes with Verdier duality, it suffices to prove the lemma under the assumption that 
%$\mathrm{rk}_{\mathrm{ss}}(M) \leq \mathrm{rk}_{\mathrm{ss}}(L)$, where $\mathrm{rk}_{\mathrm{ss}}(-)$ is 
the semisimple rank of $M$ is less than or equal to the semisimple rank of $L$.

First, if $L$ and $M$ are $G$-conjugate, then Proposition~\ref{prop:unicity-cuspidal} implies that $\cE_L$ and $\cE_M$ have distinct central characters, and the vanishing follows from Lemma~\ref{lem:central-char-1}. Now we assume that $L$ and $M$ are not $G$-conjugate.
By adjunction we have
\begin{multline*}
\Hom^k_{\Db_G(\cN_G,\bk)}(\Ind_{L \subset P}^G(\IC(\cO_L,\cE_L)), \Ind_{M \subset Q}^G(\IC(\cO_M,\cE_M)))\cong \\
\Hom^k_{\Db_M(\cN_M,\bk)}({}' \Res_{M \subset Q}^G \circ \Ind_{L \subset P}^G(\IC(\cO_L,\cE_L)), \IC(\cO_M,\cE_M)).
\end{multline*}
Now we apply the `Mackey formula' for perverse sheaves on $\cN_G$, as stated in~\cite[\mgmackey]{genspring3}. With the notation of that theorem, for any $i \in \{1, \cdots, s\}$ the inclusion $M \cap {}^{g_i} L \subset {}^{g_i} L$ is strict: in fact if it is was not we would have ${}^{g_i} L \subset M$, hence ${}^{g_i} L = M$ by our assumption on semisimple ranks, which is impossible. Since $(\cO_L,\cE_L)$ is cuspidal, this implies that all the subquotients in the filtration vanish, hence that ${}' \Res_{M \subset Q}^G \circ \Ind_{L \subset P}^G(\IC(\cO_L,\cE_L))=0$. In turn, this clearly implies the desired vanishing. 
\end{proof}

\begin{proof}[Proof of Theorem~{\rm \ref{thm:decomposition-D-0series}}]
What we have to prove is that if $(L,\cO_L,\cE_L)$ and $(L',\cO_{L'},\cE_{L'})$ are non-conjugate $0$-cuspidal data, then for any $(\cO,\cE) \in \fN^{\zs(L,\cO_L,\cE_L)}_{G,\bk}$ and $(\cO',\cE') \in \fN_{G,\bk}^{\zs(L',\cO_{L'},\cE_{L'})}$ and any $k \in \Z$ we have
\begin{equation}
\label{eqn:vanishing-0series}
\Hom^k_{\Db_G(\cN_G,\bk)}(\IC(\cO,\cE), \IC(\cO',\cE'))=0.
\end{equation}
We will fix the non-conjugate $0$-cuspidal data $(L,\cO_L,\cE_L)$ and $(L',\cO_{L'},\cE_{M'})$, and prove~\eqref{eqn:vanishing-0series} by downward induction on the induction series to which $(\cO,\cE)$ and $(\cO',\cE')$ belong (using in particular Theorem~\ref{thm:0-series}). For brevity, we write simply $\Hom^k$ instead of $\Hom^k_{\Db_G(\cN_G,\bk)}$.

We consider series
$\fN_{G,\bk}^{(M,\cO_M,\cE_M)} \subset \fN_{G,\bk}^{\zs(L,\cO_L,\cE_L)}$ and $\fN_{G,\bk}^{(M',\cO_{M'},\cE_{M'})} \subset \fN_{G,\bk}^{\zs(L',\cO_{L'},\cE_{L'})}$, and assume that property~\eqref{eqn:vanishing-0series} is known for pairs which belong to series $\fN_{G,\bk}^{(N,\cO_N,\cE_N)} \subset \fN_{G,\bk}^{\zs(L,\cO_L,\cE_L)}$ and $\fN_{G,\bk}^{(N',\cO_{N'},\cE_{N'})} \subset \fN_{G,\bk}^{\zs(L',\cO_{L'},\cE_{L'})}$ such that either $(N,\cO_N,\cE_N) \succ_G (M,\cO_M,\cE_M)$ or $(N',\cO_{N'},\cE_{N'}) \succ_G (M',\cO_{M'},\cE_{M'})$ and any $k \in \Z$. (This assumption holds automatically in the case when the cuspidal datum $(M,\cO_M,\cE_M)$, resp.~$(M',\cO_{M'},\cE_{M'})$, is maximal 
with the property that $\fN_{G,\bk}^{(M,\cO_M,\cE_M)} \subset \fN^{\zs(L,\cO_L,\cE_L)}_{G,\bk}$, resp.~$\fN_{G,\bk}^{(M',\cO_{M'},\cE_{M'})} \subset \fN_{G,\bk}^{\zs(L',\cO_{L'},\cE_{L'})}$, which is the initial step of the induction.)
We also assume for a contradiction that~\eqref{eqn:vanishing-0series} does not hold for some $(\cO,\cE) \in \fN_{G,\bk}^{(M,\cO_M,\cE_M)}$, $(\cO',\cE') \in \fN_{G,\bk}^{(M',\cO_{M'},\cE_{M'})}$ and $k \in \Z$, and choose such a triple which satisfies
\[
\Hom^{k-1}(\IC(\widetilde{\cO},\widetilde{\cE}), \IC(\widetilde{\cO}',\widetilde{\cE}'))=0
\]
for any $(\widetilde{\cO},\widetilde{\cE}) \in \fN_{G,\bk}^{(M,\cO_M,\cE_M)}$ and $(\widetilde{\cO}',\widetilde{\cE}') \in \fN_{G,\bk}^{(M',\cO_{M'},\cE_{M'})}$. Choose also parabolic subgroups $P,P' \subset G$ with respective Levi subgroups $M$ and $M'$.
By definition of induction series and \cite[Lemma~2.3]{genspring2}, there exist short exact sequences of perverse sheaves
\begin{equation}
\label{eqn:ses}
\mathrm{Ker} \hookrightarrow \Ind_{M \subset P}^G(\IC(\cO_M,\cE_M)) \twoheadrightarrow \IC(\cO,\cE) \quad \text{and} \quad  \IC(\cO',\cE') \hookrightarrow \Ind_{M' \subset P'}^G(\IC(\cO_{M'},\cE_{M'})) \twoheadrightarrow \mathrm{Coker}.
\end{equation}
Consider the following part of the long exact sequence obtained by applying $\Hom(-, \Ind_{M' \subset P'}^G(\IC(\cO_{M'},\cE_{M'})))$ to the first exact sequence in~\eqref{eqn:ses}:
\begin{multline*}
\Hom^{k-1}(\mathrm{Ker},\Ind_{M' \subset P'}^G(\IC(\cO_{M'},\cE_{M'}))) \to \Hom^k(\IC(\cO,\cE), \Ind_{M' \subset P'}^G(\IC(\cO_{M'},\cE_{M'}))) \\
\to \Hom^k(\Ind_{M \subset P}^G(\IC(\cO_M,\cE_M)), \Ind_{M' \subset P'}^G(\IC(\cO_{M'},\cE_{M'}))).
\end{multline*}
Lemma~\ref{lem:Hom-vanishing-Ind} implies that the third term vanishes. Lemma~\ref{lem:ind-series-po} and Corollary~\ref{cor:ind-0-series-po-new}\eqref{it:ind-0}, together with our assumption and our choice of $k$, imply that
$\Hom^{k-1}(\cF,\cG)=0$
for any composition factor $\cF$ of $\mathrm{Ker}$ and $\cG$ of $\Ind_{M' \subset P'}^G(\IC(\cO_{M'},\cE_{M'}))$, so that
the first term also vanishes. We deduce that
\[
\Hom^k(\IC(\cO,\cE), \Ind_{M' \subset P'}^G(\IC(\cO_{M'},\cE_{M'})))=0.
\]
Now we consider the following part of the long exact sequence obtained by applying $\Hom(\IC(\cO,\cE),-)$ to the second exact sequence in~\eqref{eqn:ses}:
\[
\Hom^{k-1}(\IC(\cO,\cE),\mathrm{Coker}) \to \Hom^k(\IC(\cO,\cE),\IC(\cO',\cE')) \to \Hom^k(\IC(\cO,\cE),\Ind_{M' \subset P'}^G(\IC(\cO_{M'},\cE_{M'}))).
\]
We have seen that the third term vanishes, and the same arguments as above imply that the first term also vanishes. But the middle term is non-zero by assumption, which provides a contradiction and concludes the proof.
\end{proof}

%--------------------------------------------------------------------
\subsection{Semisimplicity criterion}
\label{ss:semisimplicity}
%--------------------------------------------------------------------

We conclude with the following semisimplicity criterion. %stated for arbitrary characteristic.

\begin{prop}\label{prop:semisimple-crit}
Let $\bk$ be any field satisfying~\eqref{eqn:definitely-big-enough}.
The category $\Perv_G(\cN_G,\bk)$ is semisimple if and only if $\ell \nmid |W|$.
\end{prop}

\begin{proof}
If $\ell$ divides $|W|$, then the abelian category of $\bk$-representations of $W$ is not semisimple. Since this category is a quotient of $\Perv_G(\cN_G,\bk)$ by a Serre subcategory (see~\cite[Corollary~5.2]{wgasp}), it follows that $\Perv_G(\cN_G,\bk)$ is not semisimple either.

Now, assume that $\ell$ does not divide $|W|$. Then $\ell$ is rather good by~\cite[\mgeasyrathergood]{genspring3}, and Conjecture~\ref{conj:clean} is true for $G$ and its Levi subgroups by Proposition~\ref{prop:clean-lusztig}. What we have to prove is that
\begin{equation}
\label{eqn:Ext1-vanishing}
\Ext^1_{\Perv_G(\cN_G,\bk)}(\IC(\cO,\cE), \IC(\cO',\cE'))=0
\end{equation}
for all pairs $(\cO,\cE)$ and $(\cO',\cE')$ in $\fN_{G,\bk}$. Let $(L,\cO_L,\cE_L)$, resp.~$(L',\cO_{L'}, \cE_{L'})$ be a $0$-cuspidal datum such that $(\cO,\cE) \in \fN_{G,\bk}^{\zs(L,\cO_L, \cE_L)}$, resp.~$(\cO',\cE') \in \fN_{G,\bk}^{\zs(L',\cO_{L'}, \cE_{L'})}$, and let $P$, resp.~$P'$, be a parabolic subgroup of $G$ with Levi factor $L$, resp.~$L'$. By Proposition~\ref{prop:0cusp-supercusp}, $\IC(\cO,\cE)$ is a composition factor of $\Ind_{L \subset P}^G(\IC(\cO_L,\cE_L))$. Since $\ell \nmid |N_G(L)/L|$ (as follows e.g.~from~\cite[Eq.~(4.1)]{genspring3}), this perverse sheaf is semisimple by~\cite[Remark~7.2]{genspring3}, so that $\IC(\cO,\cE)$ is in fact a direct summand of $\Ind_{L \subset P}^G(\IC(\cO_L,\cE_L))$. Similarly, $\IC(\cO',\cE')$ is a direct summand of $\Ind_{L' \subset P'}^G(\IC(\cO_{L'},\cE_{L'}))$. Hence
to prove~\eqref{eqn:Ext1-vanishing} we only have to prove that
\begin{equation}
\Ext^1_{\Perv_G(\cN_G,\bk)}(\Ind_{L \subset P}^G(\IC(\cO_L,\cE_L)), \Ind_{L' \subset P'}^G(\IC(\cO_{L'},\cE_{L'})))=0.
\end{equation}
Now by Proposition~\ref{prop:zcusp-vanish} the perverse sheaf $\IC(\cO_L,\cE_L)$ is projective in $\Perv_L(\cN_L,\bk)$. Hence
\begin{multline*}
\Ext^1_{\Perv_G(\cN_G,\bk)}(\Ind_{L \subset P}^G(\IC(\cO_L,\cE_L)), \Ind_{L' \subset P'}^G(\IC(\cO_{L'},\cE_{L'}))) \\
\cong \Ext^1_{\Perv_L(\cN_L,\bk)}(\IC(\cO_L,\cE_L), \Res_{L \subset P}^G \circ \Ind_{L' \subset P'}^G(\IC(\cO_{L'},\cE_{L'}))) =0,
\end{multline*}
which finishes the proof.
\end{proof}

\appendix

%%%%%%%%%%%%%%%%%%%%%%%%%%%%%%%%%%%%%%%%%%%%
\section{Central characters}
\label{app:central-char}
%%%%%%%%%%%%%%%%%%%%%%%%%%%%%%%%%%%%%%%%%%%%

In this appendix we collect
some standard facts concerning central characters of objects in an equivariant derived category. Some of these facts are well known and were used in~\cite{genspring2}, but since we could not find a convenient reference for the versions needed in the present paper, we include proofs. See also~\cite[Appendix A]{rr} for related results. 

Let $\F$ be an arbitrary field. Let $H$ be a connected algebraic group, and let $X$ be an $H$-variety. Let $\For : \Db_H(X, \F) \to \Db(X,\F)$ be the forgetful functor. Then if $a,p : H \times X \to X$ are the action and the projection, for any $\cF$ in $\Db_H(X,\F)$ there exists a canonical isomorphism 
\begin{equation}
\label{eqn:isom-equiv}
a^* \For(\cF) \simto p^*\For(\cF).
\end{equation}
Indeed, the inverse image $p^*$ induces an equivalence of categories $\varphi_p : \Db(X,\F) \simto \Db_H(H \times X, \F)$. Hence the object $a^* \cF$ is isomorphic to $\varphi_p(\cG)$ for a unique object $\cG$ of $\Db(X,\F)$. Taking the restriction to $\{1\} \times X$ we see that $\cG=\For(\cF)$, and taking the image under the forgetful functor $\Db_H(H \times X, \F) \to \Db(H \times X,\F)$ we deduce~\eqref{eqn:isom-equiv}.

Now, let $Z \subset H$ be a closed subgroup, and assume that $Z$ acts trivially on $X$.
Again let $\cF$ be an object of $\Db_H(X,\F)$. Taking the restriction of~\eqref{eqn:isom-equiv} to $\{z\} \times X \cong X$ for all $z \in Z$, we obtain a (functorial) action of $Z$ on the object $\For(\cF)$. By standard arguments, this action factors through $Z/Z^\circ$. If $\chi : Z/Z^\circ \to \F^\times$ is a character, we say that $\cF$ \emph{has $Z$-character $\chi$} if $Z$ acts on $\For(\cF)$ via $\chi$. When $Z$ is the center of $H$, we will rather say that $\cF$ has \emph{central character} $\chi$.

\begin{lem}
\label{lem:central-char-1}
Let $\chi, \chi'$ be distinct characters of $Z/Z^\circ$, and let $\cF$, $\cG$ be objects of $\Db_H(X,\F)$. If $\cF$ has $Z$-character $\chi$ and $\cG$ has $Z$-character $\chi'$, then $\Hom^\bullet_{\Db_H(X,\F)}(\cF, \cG)=0$.
\end{lem}

\begin{proof}
By standard arguments,
there exists a spectral sequence converging to $\Hom_{\Db_H(X,\F)}^{\bullet}(\cF, \cG)$, and with $E_2$-term
\[
\mathsf{H}^\bullet_H(\mathrm{pt}) \otimes_\F \Hom^\bullet_{\Db(X,\F)}(\For \cF, \For \cG).
\]
Therefore, it is enough to prove that $\Hom_{\Db(X,\F)}^\bullet(\For \cF, \For \cG)=0$. However, the $Z$-actions on $\For \cF$ and $\For \cG$ induce an action on $\Hom_{\Db(X,\F)}^\bullet(\For \cF, \For \cG)$, for which $Z$ acts via the character $\chi'/\chi$. On the other hand this action can be extended to $H$, as follows. For any $h \in H$, restricting isomorphism~\eqref{eqn:isom-equiv} to $\{h\} \times X$ we obtain an isomorphism $\phi^{\cF}_h : \iota_h^* (\For \cF) \simto \For \cF $, where $\iota_h : X \simto X$ is the action of $h$. We also have similar isomorphisms for $\cG$. Then we can define the action of $H$ on $\Hom_{\Db(X,\F)}^\bullet(\For \cF, \For\cG)$ by declaring that $h$ acts via the composition
\begin{multline*}
\Hom^\bullet_{\Db(X,\F)}(\For\cF, \For\cG) \xrightarrow{\iota_{h^{-1}}^*} \Hom^\bullet_{\Db(X,\F)}(\iota_{h^{-1}}^* \For\cF, \iota_{h^{-1}}^*\For\cG) \\
\xrightarrow{\phi_{h^{-1}}^{\cG} \circ (-) \circ (\phi_{h^{-1}}^{\cF})^{-1}} \Hom^\bullet_{\Db(X,\F)}(\For(\cF), \For(\cG)).
\end{multline*}
Since $H$ is connected this action is trivial, and we deduce that the $Z$-action considered above is also trivial. It follows that necessarily $\Hom_{\Db(X,\F)}^\bullet(\For \cF, \For \cG)=0$, which finishes the proof.
\end{proof}

\begin{lem}
\label{lem:central-char-Perv}
Let $\chi$ be a character of $Z/Z^\circ$.
\begin{enumerate}
\item
\label{it:central-char-Perv-1}
If $\cF$ is an object of $\Perv_H(X,\F)$ with $Z$-character $\chi$, then any subquotient of $\cF$ has $Z$-character $\chi$.
\item
\label{it:central-char-Perv-2}
Assume that the characteristic of $\F$ does not divide $|Z/Z^\circ|$. Let  $\cF$, $\cG$ be objects of $\Perv_H(X,\F)$ with $Z$-character $\chi$, and consider an exact sequence
\[
0 \to \cF \to \cH \to \cG \to 0
\]
in $\Perv_H(X,\F)$.
Then $\cH$ has $Z$-character $\chi$.
\end{enumerate}
\end{lem}

\begin{proof}
\eqref{it:central-char-Perv-1} is obvious. Let us consider~\eqref{it:central-char-Perv-2}.
For $z \in Z$, let us denote by $\phi_z : \cH \simto \cH$ the action of $z$ on $\cH$. Since $z$ acts on $\cF$ and $\cG$ via $\chi(z)$, the morphism $\chi(z)^{-1} \phi_z - \id_\cH : \cH \to \cH$ factors through a morphism $\psi_z : \cG \to \cF$. For any $n \geq 1$, using the factorization $X^n-1=(1+X + \cdots + X^{n-1})(X-1)$ we obtain that $\psi_{z^n}=n \psi_z$. Since $\psi_{z'}=0$ for all $z'\in Z^\circ$ and the order of $zZ^\circ$ in $Z/Z^\circ$ is invertible in $\F$, it follows that $\psi_z=0$, which proves that $\cH$ has $Z$-character $\chi$.
\end{proof}

\begin{lem}
\label{lem:Perv-trivial-character}
Assume that $Z$ is finite and central, that the characteristic of $\F$ does not divide $|Z|$, and that $\F$ is a splitting field for $Z$. Then the forgetful functor $\Perv_{H/Z}(X,\F) \to \Perv_H(X,\F)$ identifies $\Perv_{H/Z}(X,\F)$ with the full subcategory of $\Perv_H(X,\F)$ consisting of objects with trivial $Z$-character.
\end{lem}

\begin{proof}
Since the forgetful functors $\Perv_{H/Z}(X,\F) \to \Perv(X,\F)$ and $\Perv_{H}(X,\F) \to \Perv(X,\F)$ are fully faithful, our functor $\Perv_{H/Z}(X,\F) \to \Perv_H(X,\F)$ is also fully faithful. Clearly, all the objects in its essential image have trivial $Z$-character. Conversely, let $\cF$ be an object of $\Perv_H(X,\F)$ with trivial $Z$-character. For simplicity, we denote similarly its image in $\Perv(X,\F)$. Let $a,p : H \times X \to X$ be the action and the projection respectively, and let $a',p' : H/Z \times X \to X$ be the similar morphisms for $H/Z$. Let also $\xi :  H \times X \to H/Z \times X$ be the projection, so that $a=a'\circ\xi$ and $p=p' \circ \xi$. To show that $\cF$ is $H/Z$-equivariant, we have to show that the objects $(a')^* \cF$ and $(p')^* \cF$ of $\Db(H/Z \times X, \F)$ are isomorphic. Consider the space
\begin{multline*}
\Hom_{\Db(H \times X, \F)}(a^* \cF, p^* \cF) \cong \Hom_{\Db(H \times X, \F)}(\xi^* (a')^* \cF, \xi^* (p')^* \cF) \\
\cong \Hom_{\Db(H/Z \times X, \F)}((a')^* \cF, \xi_* \xi^* (p')^* \cF) \cong \Hom_{\Db(H/Z \times X, \F)}((a')^* \cF, (p')^* \cF \otimes \xi_* \underline{\F}_{H \times X}).
\end{multline*}
Our assumptions on $\F$ imply that we have a decomposition $\xi_* \underline{\F}_{H \times X} \cong \bigoplus_\chi \cL_\chi$, where $\chi$ runs over the characters $Z \to \F^\times$, and each $\cL_\chi$ is a rank-$1$ $H$-equivariant local system with $Z$-character $\chi$. We deduce an isomorphism
\[
\Hom_{\Db(H \times X, \F)}(a^* \cF, p^* \cF) \cong \bigoplus_\chi \Hom_{\Db(H/Z \times X, \F)}((a')^* \cF, (p')^* \cF \otimes \cL_\chi).
\]
Now, since $a'$ is an $H$-equivariant morphism (where $H$ acts on $H/Z \times X$ via its action on the first factor), $(a')^* \cF$, considered as an object of $\Db_H(H/Z \times X, \F)$, has trivial $Z$-character, and it is clear that $(p')^* \cF$ also has trivial $Z$-character. We deduce that
\[
\Hom_{\Db(H/Z \times X, \F)}((a')^* \cF, (p')^* \cF \otimes \cL_\chi) =0 \qquad \text{if $\chi\neq1$}
\]
(see the proof of Lemma~\ref{lem:central-char-1}), and then that
\[
\Hom_{\Db(H \times X, \F)}(a^* \cF, p^* \cF) \cong \Hom_{\Db(H/Z \times X, \F)}((a')^* \cF, (p')^* \cF \otimes \cL_1) = \Hom_{\Db(H/Z \times X, \F)}((a')^* \cF, (p')^* \cF).
\]
Since $\cF$ is $H$-equivariant, there exists an isomorphism $a^* \cF \simto p^* \cF$. The image of this morphism under the isomorphism constructed above provides a morphism $(a')^* \cF \to (p')^* \cF$, which is easily shown to be an isomorphism. Thus, $\cF$ is $H/Z$-equivariant.
\end{proof}

Finally, we consider the setting of the body of the paper, namely the case of the $G$-action on the nilpotent cone $\cN_G$, with $Z=Z(G)$, and with coefficient field $\bk$ satisfying~\eqref{eqn:definitely-big-enough}. In particular, since $Z(G) / Z(G)^\circ$ is isomorphic to $A_G(x)$ if $x$ is a regular nilpotent element, this implies that all the irreducible $\bk$-representations of $Z(G) / Z(G)^\circ$ are characters. For any $\chi \in \Irr(\bk[Z(G)/Z(G)^\circ])$, we denote by $\Perv_G(\cN_G,\bk)_\chi$ the full subcategory of $\Perv_G(\cN_G,\bk)$ whose objects have central character $\chi$. We also denote by
$\Db_G(\cN_G,\bk)_\chi$ the full subcategory of $\Db_G(\cN_G,\bk)$ whose objects are the complexes $\cF$ such that ${}^p \hspace{-1pt} \cH^n(\cF)$ belongs to $\Perv_G(\cN_G,\bk)_\chi$ for any $n \in \Z$. It follows from Lemma~\ref{lem:central-char-Perv} that $\Perv_G(\cN_G,\bk)_\chi$ is a Serre subcategory of $\Perv_G(\cN_G,\bk)$, and that $\Db_G(\cN_G,\bk)_\chi$ is a triangulated subcategory of $\Db_G(\cN_G,\bk)$.

If $\cF$ is a simple object in $\Perv_G(\cN_G,\bk)$, then $\End(\cF)=\bk$ under our assumption~\eqref{eqn:definitely-big-enough}, so that $\cF$ has a central character. This central character can be described more explicitly as follows (see~\cite[\S 5.1]{genspring2}).
Let $(\cO,\cE) \in \fN_{G, \bk}$, and let $x \in \cO$. Then $\cE$ corresponds to an absolutely irreducible representation $V$ of $A_G(x)$. Consider the composition $Z(G) \to G_x \to A_G(x)$. This morphism is trivial on $Z(G)^\circ$, and its image is central in $A_G(x)$. Hence, by Schur's lemma, $Z(G)$ acts on $V$ via a character $\chi$ of $Z(G)/Z(G)^\circ$. Then $\IC(\cO,\cE)$ has central character $\chi$. 
%We deduce the following fact.

\begin{lem}\label{lem:centralchar-decomp}
Suppose that $\ell \nmid |Z(G)/Z(G)^\circ|$.  Then we have
\[
\Perv_G(\cN_G,\bk) = \bigoplus_{\chi \in \Irr(\bk[Z(G)/Z(G)^\circ])} \Perv_G(\cN_G,\bk)_\chi
\quad\text{and}\quad
\Db_G(\cN_G,\bk) = \bigoplus_{\chi \in \Irr(\bk[Z(G)/Z(G)^\circ])} \Db_G(\cN_G,\bk)_\chi.
\]
\end{lem}

\begin{proof}
Let us first consider the category $\Perv_G(\cN_G,\bk)$.
Lemma~\ref{lem:central-char-1} implies that any morphism and any extension between objects with distinct central characters is trivial. Since any simple object has a central character, we deduce the decomposition as stated.

To prove the decomposition for the category $\Db_G(\cN_G,\bk)$, since this category is generated by $\Perv_G(\cN_G,\bk)$ (as a triangulated subcategory), it suffices to prove that if $\chi \neq \chi'$, for $\cF$ in $\Db_G(\cN_G,\bk)_\chi$ and $\cG$ in $\Db_G(\cN_G,\bk)_{\chi'}$ we have $\Hom(\cF,\cG)=0$. This follows again from Lemma~\ref{lem:central-char-1}, using an induction on the number of nonzero perverse cohomology objects of $\cF$ and $\cG$.
\end{proof}

%%%%%%%%%%%%%%%%%%%%%%%%%%%%%%%%%%%%%%%%%%%%%%%%%%%%%%%%%%%%%%%%%%%%%%%%%%%

\end{document}